\documentclass[11pt,twoside,a4paper]{article}
\usepackage{makeidx}
\usepackage[dvips]{graphicx}
\usepackage{amscd}
\usepackage{amsmath}
\usepackage{amssymb}
\usepackage{stmaryrd}
\usepackage{latexsym}
\usepackage[latin1]{inputenc}
\usepackage[T1]{fontenc}   
\usepackage[english]{babel}
\newcommand{\tun}{\begin{picture}(5,0)(-2,-1)
\put(0,0){\circle*{2}}
\end{picture}}

\newcommand{\tdeux}{\begin{picture}(7,7)(0,-1)
\put(3,0){\circle*{2}}
\put(3,0){\line(0,1){5}}
\put(3,5){\circle*{2}}
\end{picture}}

\newcommand{\ttroisun}{\begin{picture}(15,8)(-5,-1)
\put(3,0){\circle*{2}}
\put(-0.65,0){$\vee$}
\put(6,7){\circle*{2}}
\put(0,7){\circle*{2}}
\end{picture}}
\newcommand{\ttroisdeux}{\begin{picture}(5,12)(-2,-1)
\put(0,0){\circle*{2}}
\put(0,0){\line(0,1){5}}
\put(0,5){\circle*{2}}
\put(0,5){\line(0,1){5}}
\put(0,10){\circle*{2}}
\end{picture}}

\newcommand{\tquatredeux}{\begin{picture}(15,18)(-5,-1)
\put(3,0){\circle*{2}}
\put(-0.65,0){$\vee$}
\put(6,7){\circle*{2}}
\put(0,7){\circle*{2}}
\put(0,14){\circle*{2}}
\put(0,7){\line(0,1){7}}
\end{picture}}

\newcommand{\tdun}[1]{\begin{picture}(10,5)(-2,-1)
\put(0,0){\circle*{2}}
\put(3,-2){\tiny #1}
\end{picture}}

\newcommand{\tddeux}[2]{\begin{picture}(12,5)(0,-1)
\put(3,0){\circle*{2}}
\put(3,0){\line(0,1){5}}
\put(3,5){\circle*{2}}
\put(6,-2){\tiny #1}
\put(6,3){\tiny #2}
\end{picture}}

\newcommand{\tdtroisun}[3]{\begin{picture}(20,12)(-5,-1)
\put(3,0){\circle*{2}}
\put(-0.65,0){$\vee$}
\put(6,7){\circle*{2}}
\put(0,7){\circle*{2}}
\put(5,-2){\tiny #1}
\put(9,5){\tiny #2}
\put(-5,5){\tiny #3}
\end{picture}}
\newcommand{\tdtroisdeux}[3]{\begin{picture}(12,12)(-2,-1)
\put(0,0){\circle*{2}}
\put(0,0){\line(0,1){5}}
\put(0,5){\circle*{2}}
\put(0,5){\line(0,1){5}}
\put(0,10){\circle*{2}}
\put(3,-2){\tiny #1}
\put(3,3){\tiny #2}
\put(3,9){\tiny #3}
\end{picture}}

\newcommand{\tdquatreun}[4]{\begin{picture}(20,12)(-5,-1)
\put(3,0){\circle*{2}}
\put(-0.6,0){$\vee$}
\put(6,7){\circle*{2}}
\put(0,7){\circle*{2}}
\put(3,7){\circle*{2}}
\put(3,0){\line(0,1){7}}
\put(5,-2){\tiny #1}
\put(8.5,5){\tiny #2}
\put(1,10){\tiny #3}
\put(-5,5){\tiny #4}
\end{picture}}
\newcommand{\tdquatredeux}[4]{\begin{picture}(20,20)(-5,-1)
\put(3,0){\circle*{2}}
\put(-.65,0){$\vee$}
\put(6,7){\circle*{2}}
\put(0,7){\circle*{2}}
\put(0,14){\circle*{2}}
\put(0,7){\line(0,1){7}}
\put(5,-2){\tiny #1}
\put(9,5){\tiny #2}
\put(-5,5){\tiny #3}
\put(-5,12){\tiny #4}
\end{picture}}

\input{xy}
\xyoption{all}

\newcommand{\BEQ}{\begin{equation}}     
\newcommand{\BEA}{\begin{eqnarray}}
\newcommand{\EEQ}{\end{equation}}       
\newcommand{\EEA}{\end{eqnarray}}
\newcommand{\eps}{\epsilon}
\newcommand{\R}{\mathbb{R}}
\newcommand{\C}{\mathbb{C}}
\newcommand{\Lea}{{\mathrm{Lea}}}
\newcommand{\Roo}{{\mathrm{Roo}}}
\newcommand{\II}{{\rm i}} 
\newcommand{\Sh}{\mathbf{Sh}}
\newcommand{\SkI}{{\mathrm{SkI}}}     

\newcommand{\eop}{\hfill $\Box$}   
\newcommand{\shuffle}{{\boxbar}}
\newcommand{\Id}{{\mathrm{Id}}}

\newcommand{\h}{\mathbf{H}}
\renewcommand{\H}{\mathbf{H}}
\newcommand{\T}{\mathbf{T}}
\newcommand{\F}{\mathbf{F}}
\renewcommand{\vec}[1]{\boldsymbol{#1}}
\newcommand{\FQSym}{\mathbf{FQSym}}

\newtheorem{defi}{\indent Definition}
\newtheorem{lemma}[defi]{\indent Lemma}

\newtheorem{prop}[defi]{\indent Proposition}

\newenvironment{proof}{{\bf Proof.}}{\hfill $\Box$}

\begin{document}

\vspace*{1.5cm}
\begin{center}
{\Large \bf Ordered forests, permutations and iterated integrals
}
\end{center}

\centerline{ {\bf Loïc Foissy}$^a$  and {\bf J\'er\'emie Unterberger}$^b$}

\vskip 0.5 cm
\centerline {$^a$Laboratoire de Math\'ematiques, Moulin de la Housse 
Universit\'e de Reims,} 
\centerline{ B.P.1039 , 
F -- 51687 Reims Cedex 2, France}
\vskip 0.5 cm
\centerline {$^b$Institut Elie Cartan,\footnote{Laboratoire 
associ\'e au CNRS UMR 7502} Universit\'e Henri Poincar\'e Nancy I,} 
\centerline{ B.P. 239, 
F -- 54506 Vand{\oe}uvre-l\`es-Nancy Cedex, France}

\vspace{2mm}
\begin{quote}

\renewcommand{\baselinestretch}{1.0}
\footnotesize
{We construct an explicit Hopf algebra isomorphism from the algebra of heap-ordered trees to that of quasi-symmetric
functions, generated by formal permutations, which is a lift of the natural projection of the Connes-Kreimer algebra of decorated rooted trees onto
the shuffle algebra. This isomorphism gives a universal way of lifting measure-indexed characters of the Connes-Kreimer algebra into
measure-indexed characters of the shuffle algebra, already introduced in \cite{Unterberger} in the framework  of rough path theory as the so-called
Fourier normal ordering algorithm.}

\end{quote}

\vspace{4mm}
\noindent
{\bf Keywords:} 
rough paths, H\"older continuity, Hopf algebra of decorated rooted trees, Hopf algebra of free quasi-symmetric functions, shuffle algebra.

\smallskip
\noindent
{\bf Mathematics Subject Classification (2000):} 05C05, 16W30, 60F05, 60G15, 60G18, 60H05 

\tableofcontents

\newpage

\section*{Introduction}

Let us consider $d$ regular functions $\Gamma_1,\ldots,\Gamma_d$ and $s,t \in \R$. To any word $a_1\ldots a_n$ in the letters $\{1,\ldots,d\}$, we associate
the iterated integral
$$I^{ts}_\Gamma(a_1\ldots a_n)=\int_s^t d\Gamma_{a_1}(x_1)\int_s^{x_1} d\Gamma_{a_2}(x_2) \ldots \int_s^{x_n-1} d\Gamma_{a_n}(x_n).$$
We shall see in the sequel words as decorated trunk trees, that is to say decorated rooted trees with no ramification. In other terms,
we shall identify the word $abc$ and the rooted tree $\tdtroisdeux{$a$}{$b$}{$c$}$, and so on.
This function extends into a function $\overline{I}^{ts}_\Gamma$ on decorated rooted trees. For example,
$$\overline{I}^{ts}_\Gamma(\tdtroisun{$a$}{$c$}{$b$})=\int_s^t d\Gamma_a(x_1) \int_s^{x_1} d\Gamma_b(x_2) \int_s^{x_1} d\Gamma_c(x_3),$$
see section 4.2 of the present text for more details.

Such an integral $\bar{I}^{ts}_\Gamma(\mathbb{T})$ can be decomposed as a sum of iterated integrals $I^{ts}_\Gamma(a_1\ldots a_n)$,
where $(a_1\ldots a_n)$ ranges in a certain set of words associated to the decorated rooted tree $\mathbb{T}$. For example:
$$I^{ts}_\Gamma(\tdtroisun{$a$}{$c$}{$b$})=I^{ts}_\Gamma(abc)+I^{ts}_\Gamma(acb).$$
This defines a linear map $\theta^d$ from the algebra $\h^d$ generated by the set of decorated rooted trees (that is to say the Connes-Kreimer Hopf algebra
of rooted trees) to the vector space $\Sh^d$ generated by words (the {\em shuffle algebra} on $d$ letters, see \cite{Abe}), 
and it turns out that this map is a Hopf algebra morphism, surjective and not injective. 

Let us consider a word $w=(\ell(1)\ldots \ell(n))$ in $\Sh^d$. A natural antecedent of $w$ by $\theta^d$ is the the trunk tree $\mathcal{T}$ with decorations 
given from the root to the leaf by $\ell(1),\ldots,\ell(n)$. However, this section of $\theta^d$ is a coalgebra morphism, but not an algebra morphism,
as trunk trees do not satisfy the shuflle relations. 
For any $\sigma \in \Sigma_n$, an antecedent $\mathcal{T}^\sigma$ of the permuted word $w^\sigma$ is constructed in \cite{Unterberger} 
with the help of Fubini's theorem, such that  $\mathcal{T}^{Id}$ is the trunk tree $\mathcal{T}$, and the shuffle relations are satisfied 
for the $\mathcal{T}^\sigma$. We would like in this article to reformulate this construction from an algebraic point of view, 
more precisely in terms of morphisms of Hopf algebras.\\

The initial motivation of this works comes from the theory of rough paths.
Assume $t\mapsto \Gamma(t)=(\Gamma_1(t),\ldots,\Gamma_d(t))$ is a smooth $d$-dimensional path, and let $V_1,\ldots,V_d:\mathbb{R}^d \longrightarrow 
\mathbb{R}^d$ be smooth vector fields. Then the classical Cauchy-Lipschitz theorem implies that the differential equation driven by $\Gamma$
\BEQ \label{eqdiff} dy(t)=\sum_{i=1}^d V_i(y(t)) d\Gamma_i(t) \EEQ
admits a unique solution with initial condition $y(0)=y_0$. The usual way to prove this is to show by a functional fixed-point theorem that iterated integrals
$$y_n \mapsto y_{n+1}(t)=y_0+\int_0^t \sum_i V_i(y_n(s)) d\Gamma_i(s)$$
converge when $n \rightarrow \infty$. 

Assume now that $\Gamma$ is only $\alpha$-Hölder, for some $\alpha \in (0,1)$. Then the Cauchy-Lipschitz theorem does not hold any more, because
one first needs to give a meaning to the above integrals, and in particular to the iterated integrals 
$\int_s^t d\Gamma_{i_1}(t_1) \ldots \int_s^{t_{n-1}}d\Gamma_{i_n}(t_n)$, for $n \geq 2$, $1\leq i_1,\ldots,i_n \leq d$.

The theory of rough paths, invented by T. Lyons \cite{Lyons} and further developed by V. Friz, N. Victoir \cite{Friz} and M. Gubinelli \cite{Gu},
implies the possibility to solve (\ref{eqdiff}) by a {\em redefinition of the integration} along $\Gamma$, using as an essential ingredient a {\em rough path}
$\mathbf{\Gamma}$ along $\Gamma$, see definition \ref{defiroughpath} of the present text. It is an 
essential tool in particular in the context of  integration with respect to stochastic processes or
of stochastic differential equations, when the driving process is {\em less} regular than the family
of Brownian motion\footnote{Recall that Brownian paths are $(\frac{1}{2}-\varepsilon)$-Hölder continuous for every $\varepsilon>0$.}.
Fractional Brownian motion with Hurst (or regularity) index $\alpha\in(0,1/2)$ is probably the most prominent example, 
and the main application of the above cited article \cite{Unterberger} so far
is indeed to the case of fractional Brownian motion with $\alpha\le 1/4$ \cite{Unterberger-fbm},
 for which more elementary
procedures (such as piecewise linear approximation \cite{CouQia} or the Malliavin calculus \cite{Nua} for instance) do not work any more.

The axioms of this definition can be reformulated in terms of Hopf algebras. Recall that if $H=(H,m,\Delta)$ is a Hopf algebra, then for any 
commutative algebra $A$, the set $Char_H(A)$ of algebra morphisms from $H$ to $A$ is a group for the convolution product $*$ induced
by the coproduct of $H$. With this formalism, a {\em formal rough path} may
 be seen as a a family $({\bf \Gamma}^{ts})_{t,s\in \mathbb{R}}$
of characters of the shuffle algebra $\Sh^d$, such that for any $s,t,u \in \mathbb{R}$,
${\bf \Gamma}^{ts}={\bf \Gamma}^{tu}*{\bf\Gamma}^{us}$. This family of characters is a {\em rough
path}  if it also satisfies some  Hölder continuity condition, see definition \ref{defiroughpath}, that we shall not detail here.

As explained earlier, the aim of this text is to give an algebraic frame to the construction of \cite{Unterberger}, in terms of Hopf algebra morphisms;
this will describe in a simple and explicit way {\em all} formal rough paths over $\Gamma$ by means of  algebraic tools. 
We use for this two families of combinatorial Hopf algebras. The first one is the Hopf algebra $\h$ introduced in \cite{Connes} 
for Renormalization in Quantum Field Theory.
It is based on (decorated or not) rooted trees; its product is given by commutative concatenation of rooted trees, giving rooted forests, 
and its coproduct by admissible cuts of trees, as recalled in section 2.2. We generalise this construction in section 3.1 to ordered rooted forests, 
that is to say rooted forests whose vertices are totally ordered. The obtained Hopf algebra $\h_o$ is neither commutative nor cocommutative.
If the total order of the vertices of the ordered forest $\mathbb{F}$ is compatible with the oriented graph structure of $\mathbb{F}$,
we shall say that $\mathbb{F}$ is {\em heap-ordered}. The set of heap-ordered forests generates a Hopf subalgebra $\h_{ho}$ of $\h_o$.
All these constructions are also generalized to decorated rooted forests.

On the other side, working with permutations instead of words,  we obtain a Hopf algebra structure on the vector space generated
by the elements of all symmetric groups $\Sigma_n$. This object, first introduced by C. Malvenuto and Ch. Reutenauer \cite{Malvenuto},
is known as the Hopf algebra of free quasi-symmetric functions $\FQSym$, because of its numerous relations with the Hopf algebra
of symmetric functions, see \cite{Duchamp}. 

We construct in section 3.1 a Hopf algebra morphism from $\h_o$ to $\FQSym$ using the notion of forest-order-preserving symmetries, 
see definition \ref{defifops}. When restricted to $\h_{ho}$, this Hopf algebra morphism $\Theta$ becomes an isomorphism. It is then natural
to consider the inverse image  of $\sigma \in \Sigma_n$ by $\Theta$: this element of $\h_{ho}$ is denoted by $\mathbb{T}^{\sigma^{-1}}$.
The product and coproduct of the elements $\mathbb{T}^\sigma$ is decribed in lemma \ref{l3}.

There exist canonical projections from a decorated version of $\FQSym$ to the shuffle algebra $\Sh^d$, and from the Hopf algebra
of heap-ordered decorated forests $\h_{ho}^d$ to the Hopf algebra of decorated rooted trees $\h^d$. Completing the last edge, we define a commutative square
of Hopf algebra morphisms:
$$\xymatrix{\h_{ho}^d \ar[r]^{\Theta^d} \ar[d]_{\pi_{ho}^d}&\FQSym^d \ar[d]^{\pi_\Sigma^d}\\
\h^d \ar[r]_{\theta^d}&\Sh^d}$$
Considering characters, the group of characters $Char_{\h^d}(A)$ is now seen as a subgroup of $Char_{\h_{ho}^d}(A)$, more precisely, as 
the subgroup of characters invariant under forest-order-preserving symmetries.
The Hopf algebra morphism $\theta^d$ induces a group injection from $Char_{\Sh^d}(A)$ to $Char_{\h^d}(A)$, sending $\phi$ to
$\phi\circ \theta^d$. Using characters given by iterated integrals, this allows to  compute easily
 the elements $\mathbb{T}^\sigma$
with the help of Fubini's theorem, see section 4.2.

The final section explains how to use this formalism to construct, first characters of the shuffle algebra from 
a character of $\h^d$, then a rough path, using the notion of {\em measure splitting}, see 
Definition \ref{defimeassplit}. In particular, constructing a {\em formal rough path} over $\Gamma$ 
is definitely a {\em very undetermined} problem, since essentially any choice of function on the
set of rooted trees yields by linear and multiplicative extension a character of ${\bf H}^d$ and
then a formal rough path. However, it is natural to look for some non-arbitrary regularization procedure
yielding a H\"older-continuous rough path.
As explained in the conclusion of the present text, the a priori difficult 
 problem of regularizing characters of $\Sh^d$ becomes in this process
the problem of regularizing characters of $\h^d$. This is a much more conventional problem, which
may for instance be solved by using the by now classical tools of renormalization theory \cite{Unt-ren}. \\

{\bf Remark.} The base field is $K=\mathbb{R}$ or $\mathbb{C}$. 

\section{Words and rooted trees}

\subsection{The shuffle algebra}

Let $d \geq 1$. A {\it $d$-word} is a finite sequence of elements  taken in $\{1,\ldots,d\}$. The degree of a word is the number of its letters.
In particular, there exists only one word of degree $0$, the empty word, denoted by $1$. \\

The {\it shuffle Hopf algebra} $\Sh^d$ is, as a vector space, generated by the set of $d$-words. The product $\boxbar$ of $\Sh^d$
is given in the following way: if $w$ is a $d$-word of degree $k$, $w'$ is a $d$-word of degree $l$, then
$$w\boxbar w'=\sum_{w'' \in Sh(w,w')} w'',$$
where $Sh(w,w')$ is the set of all words obtained by shuffling the letters of $w$ and $w'$. For example, 
if $(a_1a_2a_3)$ and $(a_4a_5)$ are two $d$-words (that is to say $1 \leq a_1,a_2,a_3,a_4,a_5 \leq d$):
\begin{eqnarray*}
(a_1a_2a_3)\boxbar (a_4a_5)&=&(a_1a_2a_3a_4a_5)+(a_1a_2a_4a_3a_5)+(a_1a_2a_4a_5a_3)\\
&&+(a_1a_4a_2a_3a_5)+(a_1a_4a_2a_5a_3)+(a_1a_4a_5a_2a_3)\\
&&+(a_4a_1a_2a_3a_5)+(a_4a_1a_2a_5a_3)+(a_4a_1a_5a_2a_3)+(a_4a_5a_1a_2a_3).
\end{eqnarray*}
This product is commutative; the unit is the empty word $1$. The coproduct is defined on any $d$-word $w=(a_1\ldots a_n)$ by:
$$\Delta(w)=\sum_{i=0}^n (a_1\ldots a_i) \otimes (a_{i+1}\ldots a_n).$$
For example:
\begin{eqnarray*}
\Delta(a_1a_2a_3a_4)&=&a_1a_2a_3a_4 \otimes 1+a_1a_2a_3 \otimes a_1+a_1a_2 \otimes a_3a_4+a_1 \otimes a_2a_3a_4 
+1\otimes a_1a_2a_3a_4.
\end{eqnarray*}
The counit sends $1$ to $1$ and any non-empty word to $0$. The antipode $S$ sends the word $(a_1\ldots a_n)$ to $(-1)^n (a_n\ldots a_1)$.\\

We shall consider in the sequel $d$-words as trunk trees. For example, we shall identify the $d$-word $(abc)$ with the trunk tree $\tdtroisdeux{$a$}{$b$}{$c$}$.
Considering $d$-words as trunk trees, $\Sh^d$ becomes a vector subspace and a sub-coalgebra (but not a subalgebra) of $\H^d$ 
whose definition we shall now recall.

\subsection{Reminders on rooted trees and forests}

A {\it rooted tree} is a finite tree with a distinguished vertex called the root \cite{Stanley}. 
A rooted forest is a finite graph $\mathcal{F}$ such that any connected component of $\mathcal{F}$ is a rooted tree.
The set of vertices of the rooted forest $\mathcal{F}$ is denoted by $V(\mathcal{F})$.

Let $\mathcal{F}$ be a rooted forest. The edges of $\mathcal{F}$ are oriented downwards (from the leaves to the roots). If $v,w \in V(\mathcal{F})$, with  $v\neq w$,
we shall denote $v \rightarrow w$ if there is an edge in $\mathcal{F}$ from $v$ to $w$ and $v \twoheadrightarrow w$ 
if there is an oriented path from $v$ to $w$ in $\mathcal{F}$.

Let $\vec{v}$ be a subset of $V(\mathcal{F})$. We shall say that $\vec{v}$ is an admissible cut of $\mathcal{F}$, 
and we shall write $\vec{v} \models V(\mathcal{F})$, if $\vec{v}$ is totally disconnected, that is to say that
$v \twoheadrightarrow w \hspace{-.7cm} / \hspace{.7cm}$ for any couple $(v,w)$ of two different elements of $\vec{v}$.
If $\vec{v} \models V(\mathcal{F})$, we denote by $Lea_{\vec{v}}\mathcal{F}$ the rooted sub-forest of $\mathcal{F}$ obtained by keeping
only the vertices above $\vec{v}$, that is to say $\{ w \in V(\mathcal{F}), \: \exists v \in \vec{v}, \:w \twoheadrightarrow v \}\cup \vec{v}$.
We denote by $Roo_{\vec{v}}\mathcal{F}$ the rooted sub-forest obtained by keeping the other vertices.

Connes and Kreimer proved in \cite{Connes} that the vector space $\h$ generated by the set of rooted forests is a Hopf algebra. 
Its product is given by the disjoint union of rooted forests, and the coproduct is defined for any rooted forest $\mathcal{F}$ by:
$$\Delta(\mathcal{F})=\sum_{\vec{v} \models V(\mathcal{F})} Roo_{\vec{v}}\mathcal{F} \otimes Lea_{\vec{v}}\mathcal{F}.$$
For example:
$$\Delta\left(\tquatredeux\right)=\tquatredeux \otimes 1+1\otimes \tquatredeux
+\ttroisun \otimes \tun+\tdeux \otimes \tdeux+\ttroisdeux \otimes \tun
+\tdeux \otimes \tun\tun+\tun \otimes \tdeux\tun.$$
The antipode $\bar{S}$  is inductively defined by:
\begin{eqnarray*}
\bar{S}(1)&=&1,\\
\bar{S}(\mathcal{F})&=&-\mathcal{F}-\sum_{\substack{\vec{v} \models V(\mathcal{F}) \\ Roo_{\vec{v}}\mathcal{F} \neq \mathcal{F},
Lea_{\vec{v}}\mathcal{F} \neq \mathcal{F}}} Roo_{\vec{v}}\mathcal{F}\  \bar{S}(Lea_{\vec{v}}\mathcal{F}).
\end{eqnarray*}

This construction is easily generalised to $d$-decorated rooted forests. A $d$-decorated forest is a couple $(\mathcal{F},\ell)$,
 where $\mathcal{F}$ is a rooted forest and $\ell$ a map from $V(\mathcal{F})$ to $\{1,\ldots,d\}$. If $\mathcal{F}$ and $\mathcal{G}$ are two $d$-decorated forests,
then $\mathcal{F} \mathcal{G}$ is naturally $d$-decorated. For any $\vec{v}\models V(\mathcal{F})$, $Lea_{\vec{v}}\mathcal{F}$ 
and $Roo_{\vec{v}}\mathcal{F}$ are also $d$-decorated by restriction, so the vector space $\h^d$ generated by $d$-decorated rooted forests is a Hopf algebra. \\

We already mentioned that $\Sh^d$ is a sub-coalgebra of $\H^d$. We shall see later on in subsection 4.2 that $\Sh^d$
may also be seen as a quotient Hopf algebra of $\H^d$, which accounts for the notation $\bar{S}$.

\section{Ordered rooted trees and permutations}

We shall here generalize the construction of product and the coproduct of $\h^d$ to the space generated by ordered rooted forests.

\subsection{Hopf algebra of ordered trees}

\begin{defi} 
An ordered (rooted) forest is a rooted forest with a total order on the set of its vertices. The set of ordered forests will be denoted by $\F_o$;
for all $n \geq 0$, the set of ordered forests with $n$ vertices will be denoted by $\F_o(n)$. 
An ordered (rooted) tree is a connected ordered forest. The set of ordered trees will be denoted by $\T_o$; for all $n \geq 1$, the set of ordered trees
with $n$ vertices will be denoted by $\T_o(n)$. The $K$-vector space generated by $\F_o$ is denoted by
$\h_o$. It is a graded space, the homogeneous component of degree $n$ being $Vect(\F_o(n))$ for all $n \in \mathbb{N}$.
\end{defi}

For example:
\begin{eqnarray*}
\T_o(1)&=&\{\tdun{1}\},\\
\T_o(2)&=&\{\tddeux{1}{2},\tddeux{2}{1}\},\\
\T_o(3)&=&\left\{\tdtroisun{1}{2}{3},\tdtroisun{2}{1}{3},\tdtroisun{3}{1}{2},
\tdtroisdeux{1}{2}{3},\tdtroisdeux{1}{3}{2},\tdtroisdeux{2}{1}{3},\tdtroisdeux{2}{3}{1},\tdtroisdeux{3}{1}{2},\tdtroisdeux{3}{2}{1}\right\};\\ \\
\F_o(0)&=&\{1\},\\
\F_o(1)&=&\{\tdun{1}\},\\
\F_o(2)&=&\{\tdun{1}\tdun{2},\tddeux{1}{2},\tddeux{2}{1}\},\\
\F_o(3)&=&\left\{\tdun{1}\tdun{2}\tdun{3},
\tdun{1}\tddeux{2}{3},\tdun{1}\tddeux{3}{2},\tdun{2}\tddeux{1}{3},\tdun{2}\tddeux{3}{1},\tdun{3}\tddeux{1}{2},\tdun{3}\tddeux{2}{1},
\tdtroisun{1}{3}{2},\tdtroisun{2}{3}{1},\tdtroisun{3}{2}{1},
\tdtroisdeux{1}{2}{3},\tdtroisdeux{1}{3}{2},\tdtroisdeux{2}{1}{3},\tdtroisdeux{2}{3}{1},\tdtroisdeux{3}{1}{2},\tdtroisdeux{3}{2}{1}\right\}.
\end{eqnarray*}

If $\mathbb{F}$ and $\mathbb{G}$ are two ordered forests, then the rooted forest $\mathbb{FG}$ is also an ordered forest with, 
for all $v \in V(\mathbb{F})$, $w \in V(\mathbb{G})$, $v<w$. This defines a non-commutative product on the the set of ordered forests. 
For example, the product of $\tdun{1}$ and $\tddeux{1}{2}$ gives $\tdun{1}\tddeux{2}{3}$, whereas the product of $\tddeux{1}{2}$ and $\tdun{1}$
gives $\tddeux{1}{2}\tdun{3}=\tdun{3}\tddeux{1}{2}$. This product is linearly extended to $\h_o$, which in this way becomes a graded algebra.\\

If $\mathbb{F}$ is an ordered forest, then any subforest of $\mathbb{F}$ is also ordered. 
So we can define a coproduct $\Delta:\h_o \longmapsto \h_o\otimes \h_o$ on $\h_o$ in the following way: for all $\mathbb{F} \in \F_o$,
$$\Delta(\mathbb{F})=\sum_{\vec{v} \models V(\mathbb{F})} Roo_{\vec{v}}\mathbb{F}\otimes Lea_{\vec{v}}\mathbb{F}.$$
As for the Connes-Kreimer Hopf algebra of rooted trees \cite{Connes}, one can prove that this coproduct is coassociative, 
so $\h_o$ is a graded Hopf algebra. For example:
$$\Delta\left(\tdquatredeux{2}{3}{4}{1}\right)=\tdquatredeux{2}{3}{4}{1} \otimes 1+1\otimes \tdquatredeux{2}{3}{4}{1}
+\tdtroisun{1}{3}{2} \otimes \tdun{1}+\tddeux{1}{2} \otimes \tddeux{2}{1}+\tdtroisdeux{2}{3}{1} \otimes \tdun{1}
+\tddeux{1}{2} \otimes \tdun{1}\tdun{2}+\tdun{1} \otimes \tddeux{3}{1}\tdun{2}.$$

\subsection{Hopf algebra of heap-ordered trees}

\begin{defi} \cite{Grossman}
An ordered forest is {\it heap-ordered} if for all $i,j \in V(\mathbb{F})$, $(i\twoheadrightarrow j)$ $\Longrightarrow$ $(i>j)$. 
The set of heap-ordered forests will be denoted by $\F_{ho}$; for all $n \geq 0$, the set of heap-ordered forests with $n$ vertices will be denoted by $\F_{ho}(n)$.
A heap-ordered tree is a connected heap-ordered forest. The set of heap-ordered trees will be denoted by $\T_{ho}$; for all $n \geq 1$, 
the set of heap-ordered trees with $n$ vertices will be denoted by $\T_{ho}(n)$.
\end{defi}

For example:
\begin{eqnarray*}
\T_{ho}(1)&=&\{\tdun{1}\},\\
\T_{ho}(2)&=&\{\tddeux{1}{2}\},\\
\T_{ho}(3)&=&\left\{\tdtroisun{1}{3}{2},\tdtroisdeux{1}{2}{3}\right\};\\ \\
\F_{ho}(0)&=&\{1\},\\
\F_{ho}(1)&=&\{\tdun{1}\},\\
\F_{ho}(2)&=&\{\tdun{1}\tdun{2},\tddeux{1}{2}\},\\
\F_{ho}(3)&=&\left\{\tdun{1}\tdun{2}\tdun{3},\tdun{1}\tddeux{2}{3},\tdun{2}\tddeux{1}{3},\tdun{3}\tddeux{1}{2},\tdtroisun{1}{3}{2},\tdtroisdeux{1}{2}{3}\right\}.
\end{eqnarray*}

If $\mathbb{F}$ and $\mathbb{G}$ are two heap-ordered forests, then $\mathbb{FG}$ is also heap-ordered. If $\mathbb{F}$ is a heap-ordered forest, 
then any subforest of $\mathbb{F}$ is heap-ordered. So the subspace $\h_{ho}$ of $\h_o$ generated by the heap-ordered forests 
is a graded Hopf subalgebra of $\h_o$.\\

Note that $\h_{ho}$ is not commutative. Indeed, $\tdun{1}.\tddeux{1}{2}=\tdun{1}\tddeux{2}{3}$ and $\tddeux{1}{2}.\tdun{1}=\tddeux{1}{2}\tdun{3}$.
It is neither cocommutative. Indeed:
$$\Delta(\tdtroisun{1}{3}{2})
=\tdtroisun{1}{3}{2}\otimes 1+1\otimes \tdtroisun{1}{3}{2}+2\tddeux{1}{2}\otimes \tdun{1}+\tdun{1}\otimes \tdun{1}\tdun{2}.$$
So neither $\h_{ho}$ nor its graded dual $\h_{ho}^*$, is isomorphic to the Hopf algebra of heap-ordered trees of \cite{Grossman,Grossman2}, 
which is cocommutative. \\

It is not difficult to generalize these constructions to decorated versions. A $d$-decorated ordered forest is a couple $(\mathbb{F},\ell)$,
where $\mathbb{F}$ is an ordered forest and $\ell$ is a map from $V(\mathbb{F})$ to $\{1,\ldots,d\}$. A $d$-decorated ordered forest will 
be denoted by $(\mathbb{F},a_1\ldots a_n)$, where $a_i$ is the value of $\ell$ on the $i$-th vertex of $\mathbb{F}$.

If $\mathbb{F}$ and $\mathbb{G}$ are two $d$-decorated ordered forests,
then $\mathbb{F} \mathbb{G}$ is naturally $d$-decorated. For any $\vec{v}\models V(\mathbb{F})$, $Lea_{\vec{v}}\mathbb{F}$ 
and $Roo_{\vec{v}}\mathbb{F}$ are also $d$-decorated by restriction, so the vector space $\h_o^d$ generated by $d$-decorated ordered forests is a Hopf algebra. 
The subspace $\h_{ho}^d$ of $\h_o^d$ generated by $d$-decorated heap-ordered forests is a Hopf subalgebra.
For example, if $1\leq a_1,a_2,a_3 \leq d$:
\begin{eqnarray*}
(\tdun{1},a_1).(\tddeux{1}{2},a_2a_3)&=&(\tdun{1}\tddeux{2}{3},a_1a_2a_3),\\
\Delta((\tdtroisun{1}{3}{2},a_1a_2a_3))&=&(\tdtroisun{1}{3}{2},a_1a_2a_3) \otimes 1+1\otimes (\tdtroisun{1}{3}{2},a_1a_2a_3)
+(\tddeux{1}{2},a_1a_2) \otimes (\tdun{1},a_3)\\
&&+(\tddeux{1}{2},a_1a_3) \otimes (\tdun{1},a_2)+(\tdun{1},a_1)\otimes (\tdun{1}\tdun{1},a_2a_3).
\end{eqnarray*}

{\bf Notations.} For all $n \geq 1$, we shall denote the trunk tree with $n$ vertices by $\mathcal{T}_n$. This tree has a unique heap-ordering, from the root
to the unique leaf. Identifying $d$-words with $d$-decorated trunk trees, $\Sh^d$ is now seen as a subspace and a sub-coalgebra of $\h_{ho}^d$.
For example, we shall identify:
$$(a_1a_2a_3)=\tdtroisdeux{$a_1$}{$a_2$}{$a_3$}\hspace{1mm}=(\mathcal{T}_3,\ell),$$
where $\ell:\{1,2,3\}\longrightarrow \{1,\ldots,d\}$ sends $i$ to $a_i$ for all $1\leq i \leq 3$.

\subsection{Hopf algebra of permutations}

{\bf Notations.} Let $k,l$ be integers. A {\it $(k,l)$-shuffle} is a permutation $\zeta$ of $\{1,\ldots,k+l\}$, such that 
$\zeta^{-1}(1)<\ldots< \zeta^{-1}(k)$ and $\zeta^{-1}(k+1)<\ldots< \zeta^{-1}(k+l)$. The set of $(k,l)$-shuffles will be denoted by $Sh(k,l)$.\\

{\bf Remarks.} \begin{enumerate}
\item We represent a permutation $\sigma \in S_n$ by the word $(\sigma(1)\ldots \sigma(n))$.
Then $Sh(k,l)$ is the set of words $Sh((1\ldots k),(k+1\ldots k+l))$, with the notations of subsection 2.1. For example,
$Sh(2,1)=\{(123),\:(132),\:(312)\}$.
\item For any integers $k,l$, any permutation $\sigma \in \Sigma_{k+l}$ can be uniquely written as $(\sigma_1\otimes \sigma_2) \circ \epsilon$, where 
$\sigma_1 \in \Sigma_k$, $\sigma_2 \in \Sigma_l$, and $\epsilon \in Sh(k,l)$. Similarly, considering the inverses,
any permutation $\tau \in \Sigma_{k+l}$ can be uniquely written as $\zeta^{-1}\circ (\tau_1\otimes \tau_2)$, where 
$\tau_1 \in \Sigma_k$, $\tau_2 \in \Sigma_l$, and $\zeta \in Sh(k,l)$. Note that, whereas $\eps$ shuffles the lists
$(\sigma(1),\ldots,\sigma(k)),(\sigma(k+1),\ldots,\sigma(k+l))$,
$\zeta^{-1}$ renames the numbers of each lists $(\tau(1),\ldots,\tau(k)),(\tau(k+1),\ldots,\tau(k+l))$ without
changing their orderings. For instance, $\{((2 1)\otimes 3)\circ\eps,\eps\in Sh(2,1)\}=\{(2 1 3), (2 3 1), (3 2 1)\}$,
whereas $\{\zeta^{-1}\circ((2 1)\otimes 3), \zeta\in Sh(2,1)\}=\{(2 1 3),(3 1 2),(3 2 1)\} $.
\end{enumerate}

We here briefly recall the construction of the Hopf algebra $\FQSym$ of free quasi-symmetric functions, also called the Malvenuto-Reutenauer Hopf algebra
\cite{Duchamp,Malvenuto}. As a vector space, a basis of $\FQSym$ is given by the disjoint union of the symmetric groups $\Sigma_n$, for all $n \geq 0$.
 By convention, the unique element of $S_0$ is denoted by $1$.
The product of $\FQSym$ is given, for $\sigma \in \Sigma_k$, $\tau \in \Sigma_l$, by:
$$\sigma.\tau=\sum_{\epsilon \in Sh(k,l)} (\sigma \otimes \tau) \circ \epsilon.$$
In other words, the product of $\sigma$ and $\tau$ is given by shifting the letters of the word
representing $\tau$ by $k$, and then summing all the possible shufflings of this word and of the word representing $\sigma$.
For example:
\begin{eqnarray*}
(123)(21)&=&(12354)+(12534)+(15234)+(51234)+(12543)\\
&&+(15243)+(51243)+(15423)+(51423)+(54123).
\end{eqnarray*}

Let $\sigma \in \Sigma_n$. For all $0\leq k \leq n$, there exists a unique triple 
$\left(\sigma_1^{(k)},\sigma_2^{(k)},\zeta_k\right)\in \Sigma_k \times \Sigma_{n-k} \times Sh(k,l)$
such that $\sigma=\zeta_k^{-1} \circ \left(\sigma_1^{(k)} \otimes \sigma_2^{(k)}\right)$. The coproduct of $\FQSym$ is then defined by:
$$\Delta(\sigma)=\sum_{k=0}^n \sigma_1^{(k)} \otimes \sigma_2^{(k)}
=\sum_{k=0}^n \sum_{\substack{\sigma=\zeta^{-1}\circ (\sigma_1 \otimes \sigma_2)\\ \zeta \in Sh(k,l), \sigma_1 \in \Sigma_k,\sigma_2 \in \Sigma_l}}
\sigma_1 \otimes \sigma_2.$$
Note that $\sigma_1^{(k)}$ and $\sigma_2^{(k)}$ are obtained by cutting the word representing $\sigma$ between the $k$-th and the $k+1$-th letter,
and then {\it standardizing} the two obtained words, that is to say applying to their letters the unique increasing bijection to $\{1,\ldots,k\}$ or $\{1,\ldots,n-k\}$.
For example:
\begin{eqnarray*}
\Delta((41325))&=&1\otimes (41325)+Std(4)\otimes Std(1325)+Std(41)\otimes Std(325)\\
&&+Std(413)\otimes Std(25)+Std(4132)\otimes Std(5)+(41325)\otimes 1\\
&=&1\otimes (41325)+(1) \otimes (1324)+(21) \otimes (213)\\
&&+(312)\otimes (12)+(4132) \otimes (1)+(41325) \otimes (1).
\end{eqnarray*}
Then $\FQSym$ is a Hopf algebra. It is graded, with $\FQSym(n)=vect(\Sigma_n)$ for all $n \geq 0$. \\

It is also possible to give a decorated version of $\FQSym$. A $d$-decorated permutation is a couple $(\sigma,\ell)$, where $\sigma \in \Sigma_n$
and $\ell$ is a map from $\{1,\ldots,n\}$ to $\{1,\ldots,d\}$. A $d$-decorated permutation is represented by two superposed words 
$\left(\substack{a_1\ldots a_n\\b_1\ldots b_n}\right)$,
where $(a_1 \ldots a_n)$ is the word representing $\sigma$ and for all $i$, $b_i=\ell(a_i)$. The vector space $\FQSym^d$ generated by the set
of $d$-decorated permutations is a Hopf algebra. For example, if $1\leq a,b,c,d \leq d$:
\begin{eqnarray*}
\left(\substack{213\\bac}\right).\left(\substack{1\\d}\right)
&=&\left(\substack{2134\\bacd}\right)+\left(\substack{2143\\badc}\right)+\left(\substack{2413\\bdac}\right)+\left(\substack{4213\\dbac}\right),\\
\Delta\left(\substack{4321\\dcba}\right)&=&\left(\substack{4321\\dcba}\right)\otimes 1
+\left(\substack{321\\dcb}\right)\otimes \left(\substack{1\\a}\right)+\left(\substack{21\\dc}\right)\otimes \left(\substack{21\\ba}\right)+
\left(\substack{1\\d}\right) \otimes \left(\substack{321\\cba}\right)+1\otimes \left(\substack{4321\\dcba}\right).
\end{eqnarray*}
In other words, if $(\sigma,\ell)$ and $(\tau,\ell')$ are decorated permutations of respective degrees $k$ and $l$:
\BEQ (\sigma,\ell).(\tau,\ell')=\sum_{\epsilon \in Sh(k,l)}((\sigma \otimes \tau)\circ \epsilon, \ell \otimes \ell'), \EEQ
where $\ell \otimes \ell'$ is defined by $(\ell \otimes \ell')(i)=\ell(i)$ if $1 \leq i \leq m$ and $(\ell \otimes \ell')(m+j)=\ell'(j)$ if $1\leq j \leq m'$.
If $(\sigma,\ell)$ is a decorated permutation of degree $n$:
\BEQ \Delta((\sigma,\ell))=\sum_{k=0}^n \sum_{\substack{\sigma=\zeta^{-1}\circ (\sigma_1 \otimes \sigma_2)\\ 
\zeta \in Sh(k,l), \sigma_1 \in \Sigma_k,\sigma_2 \in \Sigma_l}} (\sigma_1 \otimes \sigma_2, (\ell \otimes \ell') \circ \zeta). \EEQ

In some sense, a $d$-decorated permutation can be seen as a word with a total order on the set of its letters.

\section{From ordered forests to permutations}

\subsection{Construction of the Hopf algebra morphism}

\begin{defi}[forest-order-preserving symmetries] \label{defifops}
Let $n \geq 0$. For all $\mathbb{F} \in \F_o(n)$, let $S_\mathbb{F}$ be the set of permutations $\sigma \in \Sigma_n$ such that for all $1\leq i,j \leq n$,
($i \twoheadrightarrow j$) $\Longrightarrow$ ($\sigma^{-1}(i)>\sigma^{-1}(j)$). The elements of $S_{\mathbb{F}}$ are called the forest-order-preserving
symmetries.
\end{defi}

\begin{prop}
Let us define:
$$\Theta:\left\{\begin{array}{rcl}
\h_o&\longmapsto&\FQSym\\
\mathbb{F} \in \F_o&\longmapsto&\displaystyle \sum_{\sigma \in S_\mathbb{F}} \sigma.
\end{array}\right.$$
Then $\Theta$ is a Hopf algebra morphism, homogeneous of degree $0$.
\end{prop}

For example:
\begin{eqnarray*}
\Theta(\tdun{1})&=&(1),\\
\Theta(\tdun{1}\tdun{2})&=&(12)+(21),\\
\Theta(\tddeux{1}{2})&=&(12),\\
\Theta(\tdun{1}\tdun{2}\tdun{3})&=&(123)+(132)+(213)+(231)+(312)+(321),\\
\Theta(\tdun{1}\tddeux{2}{3})&=&(123)+(213)+(231),\\
\Theta(\tdun{2}\tddeux{1}{3})&=&(213)+(123)+(132),\\
\Theta(\tdun{3}\tddeux{1}{2})&=&(312)+(132)+(123),\\
\Theta(\tdtroisun{1}{3}{2})&=&(123)+(132),\\
\Theta(\tdtroisdeux{1}{2}{3})&=&(123).
\end{eqnarray*}
In particular, if $\mathbb{F}$ is the product of the trunk trees with respectively $k$ and $l$ vertices, totally ordered from their roots to their leaves, 
then $S_\mathbb{F}=Sh(k,l)$. \\

\begin{proof} Obviously, $\Theta$ is homogeneous of degree $0$. Let $\mathbb{F} \in \F_o(k)$, $\mathbb{G} \in \F_o(l)$. 
Let $\sigma \in S_{\mathbb{FG}}$. Then $\sigma$ can be uniquely written as $\sigma=(\sigma_1 \otimes \sigma_2)\circ \epsilon$,
with $\sigma_1 \in \Sigma_k$, $\sigma_2 \in \Sigma_l$, and $\epsilon \in Sh(k,l)$. 
If $i \twoheadrightarrow j$ in $\mathbb{F}$, then $i\twoheadrightarrow j$ in $\mathbb{FG}$, so:
\begin{eqnarray*}
\sigma^{-1}(i)&>&\sigma^{-1}(j)\\
\epsilon^{-1} \circ \left(\sigma_1^{-1}\otimes \sigma_2^{-1}\right)(i)&>&\epsilon^{-1} \circ \left(\sigma_1^{-1}\otimes \sigma_2^{-1}\right)(j)\\
\epsilon^{-1}\left(\sigma_1^{-1}(i)\right)&>&\epsilon^{-1}\left(\sigma_1^{-1}(j)\right)\\
\sigma_1^{-1}(i)&>&\sigma_1^{-1}(j),
\end{eqnarray*}
as $\epsilon^{-1}$ is increasing on $\{1,\ldots,k\}$. So $\sigma_1\in S_\mathbb{F}$. 
If $i \twoheadrightarrow j$ in $\mathbb{G}$, then $k+i \twoheadrightarrow k+j$ in $\mathbb{FG}$, so:
\begin{eqnarray*}
\sigma^{-1}(k+i)&>&\sigma^{-1}(k+j)\\
\epsilon^{-1} \circ \left(\sigma_1^{-1}\otimes \sigma_2^{-1}\right)(k+i)&>&\epsilon^{-1} \circ \left(\sigma_1^{-1}\otimes \sigma_2^{-1}\right)(k+j)\\
\epsilon^{-1} \left(k+\sigma_2^{-1}(i)\right)&>&\epsilon^{-1} \left(k+\sigma_2^{-1}(j)\right)\\
\sigma_2^{-1}(i)&>&\sigma_2^{-1}(j),
\end{eqnarray*}
as $\epsilon^{-1}$ is increasing on $\{k+1,\ldots,k+l\}$. So $\sigma_2 \in S_\mathbb{G}$.
Conversely, if $\sigma=(\sigma_1 \otimes \sigma_2)\circ \epsilon$,
with $\sigma_1 \in \Sigma_k$, $\sigma_2 \in \Sigma_l$, and $\epsilon \in Sh(k,l)$, the same computations shows that $\sigma \in S_{\mathbb{FG}}$. So:
$$S_{\mathbb{FG}}=\bigsqcup_{\epsilon \in Sh(k,l)} (S_\mathbb{F} \otimes S_\mathbb{G})\circ \epsilon.$$
So:
$$\Theta(\mathbb{FG})=\sum_{\epsilon \in Sh(k,l)} \sum_{\sigma_1 \in S_\mathbb{F}} \sum_{\sigma_2 \in S_\mathbb{G}}(\sigma_1 \otimes \sigma_2)\circ \epsilon=\Theta(\mathbb{F})\Theta(\mathbb{G}).$$
So $\Theta$ is an algebra morphism. \\

Let $\mathbb{F} \in \F_o(n)$ and let $\vec{v}$ be an admissible cut of $\mathbb{F}$. The vertices of $Roo_{\vec{v}}\mathbb{F}$ are $i_1<\ldots<i_k$ 
and the vertices of $Lea_{\vec{v}}\mathbb{F}$ are $j_1<\ldots <j_l$, with $k+l=n$. Let $\zeta_{\vec{v}}$ be the inverse of the permutation 
$(i_1,\ldots,i_k,j_1,\ldots,j_l)$. Note that $\zeta_{\vec{v}}$ is a $(k,l)$-shuffle. Let $\sigma_1 \in S_{Roo_{\vec{v}}\mathbb{F}}$ and 
$\sigma_2 \in S_{Lea_{\vec{v}}\mathbb{F}}$. Let us show that $\sigma=\zeta_{\vec{v}}^{-1} \circ (\sigma_1 \otimes \sigma_2) \in S_\mathbb{F}$.
If $i \twoheadrightarrow j$ in $\mathbb{F}$, then three cases are possible:
\begin{itemize}
\item $i$ and $j$ belong to $Roo_{\vec{v}}\mathbb{F}$, say $i=i_p$ and $j=i_q$. Then $i \twoheadrightarrow j$ in $Roo_{\vec{v}}\mathbb{F}$, so 
$\sigma_1^{-1}(p) > \sigma_1^{-1}(q)$. Then:
\begin{eqnarray*}
\sigma^{-1}(i)&=&\left(\sigma_1^{-1} \otimes \sigma_2^{-1}\right)\circ \zeta_{\vec{v}}(i)\\
&=&\left(\sigma_1^{-1} \otimes \sigma_2^{-1}\right)(p)\\
&=&\sigma_1^{-1}(p).
\end{eqnarray*}
Similarly, $\sigma^{-1}(j)=\sigma_1^{-1}(q)$. So $\sigma^{-1}(i)> \sigma^{-1}(j)$.
\item $i$ and $j$ belong to $Lea_{\vec{v}}\mathbb{F}$. The proof is similar.
\item $i$ belongs to $Lea_{\vec{v}}\mathbb{F}$ and $j$ belongs to $Roo_{\vec{v}}\mathbb{F}$. Then $k+1\leq \zeta_{\vec{v}}(i)\leq k+l$ and 
$1\leq \zeta_{\vec{v}}(j)\leq k$, so $\sigma^{-1}(j)\leq k<k+1 \leq \sigma^{-1}(i)$.
\end{itemize}
Conversely, let $\sigma \in S_\mathbb{F}$ and $0 \leq k \leq n$. We put $\zeta=\zeta_k$, $\sigma_1=\sigma_1^{(k)}$
and $\sigma_2=\sigma_2^{(k)}$, so that $\sigma=\zeta^{-1} \circ (\sigma_1\otimes \sigma_2)$. Let $\mathbb{G}$ be the sub-forest of $\mathbb{F}$ 
formed by the vertices $\zeta(1),\ldots,\zeta(k)$ and $\mathbb{H}$ be the sub-forest of $\mathbb{F}$ formed by the vertices $\zeta(k+1),\ldots,\zeta(k+l)$, 
with $l=n-k$. If $i$ is a vertex of $\mathbb{F}$ and $j$ is a vertex of $\mathbb{H}$ such that $i \twoheadrightarrow j$ in $\mathbb{F}$, 
then $\sigma^{-1}(i)> \sigma^{-1}(j)$. As $k+1\leq \zeta(j) \leq k+l$, $k+1 \leq \sigma^{-1}(j) \leq k+l$, so $k+1 \leq \sigma^{-1}(i) \leq k+l$ 
and $k+1 \leq \zeta(i) \leq k+l$: $i$ is a vertex of $\mathbb{H}$. As a consequence, there exists a (unique) admissible cut $\vec{v}$ such that
$\mathbb{G}=Roo_{\vec{v}}\mathbb{F}$ and $\mathbb{H}=Lea_{\vec{v}}\mathbb{F}$. By definition, $\zeta=\zeta_{\vec{v}}$.
It is not difficult to prove that $\sigma_1 \in S_{Roo_{\vec{v}}\mathbb{F}}$ and $\sigma_2 \in S_{Lea_{\vec{v}}\mathbb{F}}$. 
Hence, there is a bijection:
$$\left\{\begin{array}{rcl}
S_\mathbb{F}\times \{0,\ldots,n\}&\longmapsto&\displaystyle \bigsqcup_{\vec{v} \models V(\mathbb{F})} S_{Roo_{\vec{v}}\mathbb{F}}\times S_{Lea_{\vec{v}}\mathbb{F}}\\
(\sigma,k)&\longmapsto&\left(\sigma_1^{(k)},\sigma_2^{(k)}\right).
\end{array}\right.$$
Finally:
$$\Delta\circ \Theta(\mathbb{F})=\sum_{\sigma \in S_\mathbb{F}}\sum_{k=0}^n \sigma_1^{(k)} \otimes \sigma_2^{(k)}
=\sum_{\vec{v} \models V(\mathbb{F})} \sum_{\sigma_1 \in S_{Roo_{\vec{v}}\mathbb{F}}} 
\sum_{\sigma_2 \in S_{Lea_{\vec{v}}\mathbb{F}}} \sigma_1 \otimes \sigma_2
=(\Theta \otimes \Theta) \circ \Delta(\mathbb{F}).$$
So $\Theta$ is a coalgebra morphism. \end{proof}

\subsection{Restriction to heap-ordered forests}

\begin{prop} \label{p2}
The restriction of $\Theta$ to $\h_{ho}$ is an isomorphism of graded Hopf algebras.
\end{prop}

\begin{proof}  As $\Theta$ is homogenous, $\Theta(\h_{ho}(n)) \subseteq \FQSym(n)$ for all $n \geq 0$. 
Let us first recall that $dim(\h_{ho}(n))=n!$. From Lemma 6.5 of \cite{Grossman2}, the number of heap-ordered trees with $n+1$ vertices is $n!$.
This is proved inductively, using the bijection:
$$\left\{\begin{array}{rcl}
\T_{ho}(n-1)\times \{1,\ldots,n-1\}&\longmapsto&\T_{ho}(n)\\
(t,i)&\longmapsto&\mbox{the heap-ordered tree obtained}\\
&&\mbox{ by grafting $n$ on the vertex $i$ of $t$}.
\end{array}\right.$$
If $t$ is a heap-ordered tree, then its root is its smallest element, so there is a bijection:
$$\left\{\begin{array}{rcl}
\T_{ho}(n+1)&\longmapsto &\F_{ho}(n)\\
t&\longmapsto&\mbox{the heap-ordered forest obtained by deleting the root of $t$}.
\end{array}\right.$$
So $card(\F_{ho}(n))=n!=dim(\h_{ho}(n))$. 

In order to prove that $\Theta_{\mid \h_{ho}}$ is an isomorphism, it is now enough to prove that $\Theta(\h_{ho}(n))=\FQSym(n)$.
We totally order the elements of $\Sigma_n$ by the lexicographic order. Let us prove that for any $\sigma \in \Sigma_n$,
there exists a heap-ordered forest $\mathbb{F}$ such that $\sigma$ is the greatest element of $S_\mathbb{F}$. 
This will imply that $\Theta(\h_{ho}(n))=\FQSym(n)$. If $\sigma=(1,\ldots,n)$ and $\mathbb{T}_n$ is the trunk tree with $n$ vertices, 
totally ordered from the root to the leaf, then $\mathbb{T}_n$ is heap-ordered and $\Theta(S_{\mathbb{T}_n})=\{(1,\ldots,n)\}$.
Let us take $\sigma\neq(1,\ldots,n)$ and let us assume that the result is true for any $\tau<\sigma$. Let $i$ be the smallest
index such that $\sigma(i)\neq i$. We put $j=\sigma(i)$, then $j>i$, so $j\geq 2$. Then, denoting $\tau$ the transposition permuting $j-1$ and $j$,
$\tau \circ \sigma=(1,\ldots,i-1,j-1,\ldots)<\sigma$. Let $\mathbb{F} \in \F_{ho}(n)$, such that the greatest element of $S_\mathbb{F}$ is $\tau \circ \sigma$.
We consider the vertices $i-1$ and $i$ of $\mathbb{F}$. As $\mathbb{F}$ is heap-ordered, two cases are possible:
\begin{itemize}
\item $i-1$ and $i$ are independent vertices of $\mathbb{F}$, that is to say $i \twoheadrightarrow  \hspace{-.45cm} / \hspace{.3cm} i-1$ and
$i-1 \twoheadrightarrow  \hspace{-.45cm} / \hspace{.3cm} i$. Then the ordered forest $\mathbb{F}'$ obtained from $\mathbb{F}$ 
by permuting the order of the vertices $i-1$ and $i$ is also heap-ordered, and $S_{\mathbb{F}'}=\tau\circ S_\mathbb{F}$. As a consequence, 
the greatest element of $S_{\mathbb{F}'}$ is $\sigma$. 
\item $i \rightarrow i-1$ in $\mathbb{F}$. If $i-1$ is not a root of $\mathbb{F}$, is has a direct ascendant $k$, 
then we locally apply the following local transformation on $\mathbb{F}$ to obtain a new heap-ordered forest $\mathbb{F}'$:
$$\includegraphics[height=2cm]{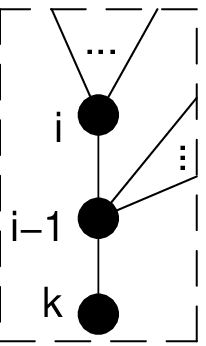}\:\longmapsto \:\includegraphics[height=2cm]{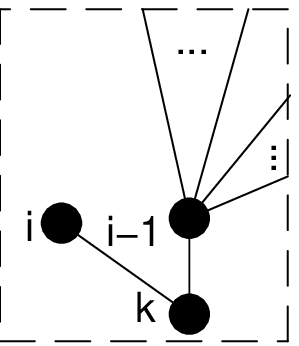}.$$
if $i$ is a root of $\mathbb{F}$, then we locally apply the following local transformation on $\mathbb{F}'$ to obtain a new heap-ordered forest $\mathbb{F}'$:
$$\includegraphics[height=1.6cm]{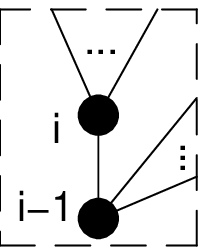}\:\longmapsto\:\includegraphics[height=1.6cm]{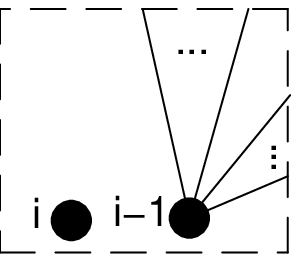}.$$
For example, if $\mathbb{F}=\tdquatredeux{1}{4}{2}{3}$ and $i=2$, $\mathbb{F}'=\tdun{2}\tdtroisun{1}{4}{3}$ ; 
if $i=3$, then $\mathbb{F}'=\tdquatreun{1}{4}{3}{2}$.

\end{itemize}
It is not difficult to see that the greatest element of $S_{\mathbb{F}'}$ is $\sigma$. As a conclusion, $\Theta_{\mid \h_{ho}}$ is an isomorphism. \end{proof}\\

\begin{defi}
Let  $\sigma \in \Sigma_n$. There exists a unique element $\mathbb{T}^\sigma \in \h_{ho}$, such that 
$\Theta(\mathbb{T}^\sigma)=\sigma^{-1}$. 
\end{defi}

\begin{lemma} \label{l3}
\begin{enumerate}
\item For any $(\sigma,\tau) \in \Sigma_k \times \Sigma_l$, 
\BEQ \mathbb{T}^\sigma \mathbb{T}^\tau=\sum_{\zeta \in Sh(k,l)} \mathbb{T}^{\zeta^{-1} \circ (\sigma \otimes \tau)}.\EEQ
\label{eq:TT1}
\item For any $\sigma \in \Sigma_n$,
\BEQ \Delta(\mathbb{T}^\sigma)=\sum_{k=0}^n \sum_{\substack{\sigma=(\sigma_1 \otimes \sigma_2)\circ \epsilon\\
\sigma_1 \in \Sigma_k,\: \sigma_2 \in \Sigma_{n-k},\: \epsilon \in Sh(k,n-k)}}
\mathbb{T}^{\sigma_1} \otimes \mathbb{T}^{\sigma_2}.  \label{eq:TT2} \EEQ
\end{enumerate}
\end{lemma}

\begin{proof} 1. Indeed:
$$\Theta(\mathbb{T}^\sigma \mathbb{T}^\tau)=\Theta(\mathbb{T}^\sigma)\Theta(\mathbb{T}^\tau)=\sigma^{-1}\tau^{-1}
=\sum_{\epsilon \in Sh(k,l)} (\sigma^{-1} \otimes \tau^{-1})\circ \epsilon
=\sum_{\zeta \in Sh(k,l)} \Theta\left(\mathbb{T}^{\zeta^{-1} \circ (\sigma \otimes \tau)}\right).$$
We conclude with the injectivity of $\Theta_{\mid \h_{ho}}$. \\

2. Indeed:
\begin{eqnarray*}
(\Theta \otimes \Theta)\circ \Delta(\mathbb{T}^{\sigma})&=&\Delta(\sigma^{-1})\\
&=&\sum_{k=0}^n \sum_{\substack{\sigma^{-1}=\zeta^{-1}\circ (\tau_1 \otimes \tau_2)\\
\tau_1 \in \Sigma_k,\: \tau_2 \in \Sigma_{n-k},\: \zeta \in Sh(k,n-k)}} \tau_1 \otimes \tau_2\\
&=&\sum_{k=0}^n \sum_{\substack{\sigma=(\sigma_1 \otimes \sigma_2)\circ \epsilon\\
\sigma_1 \in \Sigma_k,\: \sigma_2 \in \Sigma_{n-k},\: \epsilon \in Sh(k,n-k)}} \sigma_1^{-1} \otimes \sigma_2^{-1}\\
&=&(\Theta \otimes \Theta)\left(\sum_{k=0}^n \sum_{\substack{\sigma=(\sigma_1 \otimes \sigma_2)\circ \epsilon\\
\sigma_1 \in \Sigma_k,\: \sigma_2 \in \Sigma_{n-k},\: \epsilon \in Sh(k,n-k)}}
\mathbb{T}^{\sigma_1} \otimes \mathbb{T}^{\sigma_2}\right).
\end{eqnarray*}
We conclude with the injectivity of $\Theta_{\mid \h_{ho}}\otimes \Theta_{\mid \h_{ho}}$. \end{proof}\\

For example:
\begin{eqnarray*}
\mathbb{T}^{(1)}&=&\tdun{1},\\
\mathbb{T}^{(12)}&=&\tddeux{1}{2},\\
\mathbb{T}^{(21)}&=&\tdun{1}\tdun{2}-\tddeux{1}{2},\\
\mathbb{T}^{(123)}&=&\tdtroisdeux{1}{2}{3},\\
\mathbb{T}^{(132)}&=&\tdtroisun{1}{3}{2}-\tdtroisdeux{1}{2}{3},\\
\mathbb{T}^{(213)}&=&\tdun{2}\tddeux{1}{3}-\tdtroisun{1}{2}{3},\\
\mathbb{T}^{(231)}&=&\tdun{3}\tddeux{1}{2}-\tdtroisun{1}{2}{3},\\
\mathbb{T}^{(312)}&=&\tdun{1}\tddeux{2}{3}-\tdtroisdeux{1}{2}{3}-\tdun{2}\tddeux{1}{3}+\tdtroisun{1}{3}{2},\\
\mathbb{T}^{(321)}&=&\tdun{1}\tdun{2}\tdun{3}-\tdun{3}\tddeux{1}{2}-\tdun{1}\tddeux{2}{3}+\tdtroisdeux{1}{2}{3}.
\end{eqnarray*}

So:
\begin{eqnarray*}
\mathbb{T}^{(21)} \mathbb{T}^{(1)}&=&\mathbb{T}^{(213)}+\mathbb{T}^{(312)}+\mathbb{T}^{(321)},\\
\Delta\left(\mathbb{T}^{(321)}\right)&=&1\otimes \mathbb{T}^{(321)}+\tdun{$1$}\otimes (\tdun{$1$}\tdun{$2$}-\tddeux{$1$}{$2$})
+(\tdun{$1$}\tdun{$2$}-\tddeux{$1$}{$2$}) \otimes \tdun{$1$}+\mathbb{T}^{(321)}\otimes 1\\
&=&1\otimes \mathbb{T}^{(321)}+\mathbb{T}^{(1)}\otimes \mathbb{T}^{(21)}+\mathbb{T}^{(21)} \otimes \mathbb{T}^{(1)}+\mathbb{T}^{(321)}\otimes 1.
\end{eqnarray*}

We can also give a decorated version of this result. 
If $\ell$ is an application from $\{1,\ldots,m\}$ to $\{1,\ldots,d\}$ and $\ell'$ is an application from $\{1,\ldots,m'\}$ to $\{1,\ldots,d\}$,
let $(\mathbb{T}^\sigma,\ell)$ be the element of the Hopf algebra obtained by decorating all the forests appearing in $\mathbb{T}^\sigma$ by $\ell$;
we define similarly $(\mathbb{T}^{\sigma'},\ell')$. Then it comes directly from lemma \ref{l3} that:
$$(\mathbb{T}^\sigma,\ell).(\mathbb{T}^\tau,\ell')=\sum_{\zeta \in Sh(m,m')} \left(\mathbb{T}^{\zeta^{-1} \circ (\sigma \otimes \tau)},\ell\otimes \ell'\right).$$
Moreover, for any $\sigma \in \Sigma_n$,
$$\Delta(\mathbb{T}^\sigma,\ell)=\sum_{k=0}^n \sum_{\substack{\sigma=(\sigma_1 \otimes \sigma_2)\circ \epsilon\\
\sigma_1 \in \Sigma_k,\: \sigma_2 \in \Sigma_{n-k},\: \epsilon \in Sh(k,n-k)}}
(\mathbb{T}^{\sigma_1} \otimes \mathbb{T}^{\sigma_2},\ell \circ \epsilon^{-1}).$$

\subsection{Action of the symmetric groups on $\h_o$}

The symmetric group $\Sigma_n$ acts naturally on the set of ordered forests with $n$ vertices by changing the order of the vertices according to $\sigma$. 
For example, $\sigma. \tdtroisdeux{$i$}{$j$}{$k$}=\tdtroisdeux{$\sigma(i)$}{$\sigma(j)$}{$\sigma(k)$} \hspace{.4cm}$ if 
$\{i,j,k\}=\{1,2,3\}$. This action is extended by linearity to the homogeneous component $\h_o(n)$ of degree $n$ of $\h_o$. 

Let $\mathbb{F}$ be an ordered forest of degree $n$ and let $\tau \in \Sigma_n$. For any $\sigma \in \Sigma_n$:
\begin{eqnarray*}
\sigma \in S_{\tau.\mathbb{F}}&\Longleftrightarrow& \forall i,j \in V(\tau.\mathbb{F}),\: (i\twoheadrightarrow j \mbox{ in }\tau.\mathbb{F}) 
\Longrightarrow (\sigma^{-1}(i)>\sigma^{-1}(j))\\
&\Longleftrightarrow& \forall i,j \in V(\tau.\mathbb{F}),\: (\tau(i)\twoheadrightarrow \tau(j) \mbox{ in }\tau.\mathbb{F}) 
\Longrightarrow (\sigma^{-1}\circ \tau(i)>\sigma^{-1}\circ \tau(j))\\
&\Longleftrightarrow& \forall i,j \in V(\mathbb{F}),\: (i \twoheadrightarrow j \mbox{ in }\mathbb{F}) 
\Longrightarrow (\sigma^{-1}\circ \tau(i)>\sigma^{-1}\circ \tau(j))\\
&\Longleftrightarrow&\tau^{-1} \circ \sigma \in S_{\mathbb{F}}.
\end{eqnarray*}
As a consequence, $S_{\tau.\mathbb{F}}=\tau \circ S_{\mathbb{F}}$ so, for any $\mathbb{F} \in \h_o(n)$, for any $\tau \in \Sigma_n$,
$\Theta(\tau.\mathbb{F})=\tau \circ \Theta(\mathbb{F})$.\\

The subspace $\h_{ho}(n)$ of $\h_o(n)$ is clearly not stable under the action of $\Sigma_n$. More precisely, of $\mathbb{F}$ is a heap-ordered forest
and $\sigma \in \Sigma_n$, then $\sigma.\mathbb{F}$ is heap-ordered if, and only if, $\sigma^{-1} \in S_{\mathbb{F}}$.
In particular, if $\mathbb{F}$ and $\mathbb{G}$ are heap-ordered forests with respectively $k$ and $l$ vertices, then for any $(k,l)$-shuffle $\zeta$,
$\zeta \in S_{\mathbb{FG}}$, so $\zeta^{-1}.\mathbb{FG}$ is an element of $\h_{ho}(k+l)$. As a consequence, if 
$\sigma \in \Sigma_k$, $\tau \in \Sigma_l$ and $\epsilon \in Sh(k,l)$, $\epsilon^{-1}.\mathbb{T}^\sigma\mathbb{T}^\tau \in \h_{ho}$. We compute:
$$\Theta(\epsilon^{-1}.\mathbb{T}^\sigma\mathbb{T}^\tau)=\epsilon^{-1}\circ(\sigma^{-1}\tau^{-1})
=\sum_{\zeta \in Sh(k,l)} \epsilon^{-1} \circ (\sigma^{-1} \otimes \tau^{-1})\circ \zeta
=\sum_{\zeta \in Sh(k,l)} \Theta\left(\mathbb{T}^{\zeta^{-1} \circ (\sigma \otimes \tau)\circ \epsilon}\right).$$
From the injectivity of $\Theta_{\mid \h_{ho}}$, we deduce:

\begin{lemma} \label{l8}
For any $(\sigma,\tau,\epsilon) \in \Sigma_k \times \Sigma_l \times Sh(k,l)$, 
\BEQ \epsilon^{-1}.\mathbb{T}^\sigma \mathbb{T}^\tau=\sum_{\zeta \in Sh(k,l)} \mathbb{T}^{\zeta^{-1} \circ (\sigma \otimes \tau)\circ \epsilon}.\EEQ
\end{lemma}

\section{A commuting square of Hopf algebra epimorphisms}

\subsection{Definition of the square}

For any $d$, there is an isomorphism of Hopf algebras $\Theta^d:\h_{ho}^d\longrightarrow \FQSym^d$.
There are also a natural epimorphism of Hopf algebras $\pi_{ho}^d$ from $\h_{ho}^d$ to $\h^d$, sending a $d$-decorated, ordered forest to the underlying
$d$-decorated forest, and a natural epimorphism of Hopf algebras $\pi_{\Sigma}^d$ from $\FQSym^d$ to $\Sh^d$, sending the decorated permutation
$\left(\substack{a_1\ldots a_n\\b_1\ldots b_n}\right)$ to the word $(b_1\ldots b_n)$. Composing these results, we obtain the following diagram:
$$\xymatrix{\h_{ho}^d \ar[r]^{\Theta^d} \ar[d]_{\pi_{ho}^d}&\FQSym^d \ar[d]^{\pi_\Sigma^d}\\
\h^d&\Sh^d}$$

Let us consider a $d$-decorated rooted forest $\mathcal{F}$. We give it a heap-order to obtain an element $\mathbb{F}$ of $\h_{ho}^d$,
such that $\pi_{ho}^d(\mathbb{F})=\mathcal{F}$. 
It is then not difficult to show that $\pi_{\Sigma}^d \circ \Theta^d(\mathbb{F})$ does not depend of the choice of the heap-order on $\mathcal{F}$,
so this defines a map $\theta^d:\h^d\longrightarrow \Sh^d$, making the following diagram commuting:
$$\xymatrix{\h_{ho}^d \ar[r]^{\Theta^d} \ar[d]_{\pi_{ho}^d}&\FQSym^d \ar[d]^{\pi_\Sigma^d}\\
\h^d \ar[r]_{\theta^d}&\Sh^d}$$
As $\pi_{ho}^d$ is surjective, $\theta^d$ is a morphism of Hopf algebras. For example, if $1\leq a,b,c \leq d$:
\begin{eqnarray*}
\theta^d(\tdun{$a$})&=&(a),\\
\theta^d(\tddeux{$a$}{$b$})&=&(ab),\\
\theta^d(\tdun{$a$}\tdun{$b$})&=&(ab)+(ba),\\
\theta^d(\tdtroisun{$a$}{$c$}{$b$})&=&(abc)+(acb),\\
\theta^d(\tdtroisdeux{$a$}{$b$}{$c$})&=&(abc),\\
\theta^d(\tddeux{$a$}{$b$}\tdun{$c$})&=&(abc)+(acb)+(cab),\\
\theta^d(\tdun{$a$}\tdun{$b$}\tdun{$c$})&=&(abc)+(acb)+(bac)+(bca)+(cab)+(cba).
\end{eqnarray*}

Seeing $d$-words as $d$-decorated trunk trees (see subsection 2.1), one can write:
$$\theta^d(\tdtroisun{$a$}{$c$}{$b$})=\tdtroisdeux{$a$}{$b$}{$c$}+\tdtroisdeux{$a$}{$c$}{$b$}.$$
In other words, the image of a decorated tree by $\theta^d$ is the sum of all trunk trees with same decorations, 
whose total ordering is compatible with the partial ordering of the initial tree.\\

Let us now consider a $d$-word $\ell(1)\ldots \ell(n)$ of degree $n$ , or equivalently a trunk tree $\mathcal{T}$ with decoration $\ell$. We put:
$$\mathcal{T}^\sigma=\pi_{ho}^d \circ (\Theta^d)^{-1}(\sigma^{-1},\ell)=\pi_{ho}^d(\mathbb{T}^\sigma,\ell).$$
In other words, $\mathcal{T}^\sigma$ is obtained from $\mathbb{T}^\sigma$ by decorating the $i$-th vertex of the forests in $\mathbb{T}^\sigma$ 
by $\ell(i)$ for all $i$, and then deleting the orders on the the vertices. For example, if $1\leq a,b,c \leq d$:
\begin{eqnarray*}
\tdtroisdeux{$a$}{$b$}{$c$}^{(123)}&=&\tdtroisdeux{$a$}{$b$}{$c$},\\
\tdtroisdeux{$a$}{$b$}{$c$}^{(132)}&=&\tdtroisun{$a$}{$c$}{$b$}-\tdtroisdeux{$a$}{$b$}{$c$},\\
\tdtroisdeux{$a$}{$b$}{$c$}^{(213)}&=&\tdun{$b$}\tddeux{$a$}{$c$}-\tdtroisun{$a$}{$b$}{$c$},\\
\tdtroisdeux{$a$}{$b$}{$c$}^{(231)}&=&\tdun{$c$}\tddeux{$a$}{$b$}-\tdtroisun{$a$}{$b$}{$c$},\\
\tdtroisdeux{$a$}{$b$}{$c$}^{(312)}&=&\tdun{$a$}\tddeux{$b$}{$c$}-\tdtroisdeux{$a$}{$b$}{$c$}-\tdun{$b$}\tddeux{$a$}{$c$}+\tdtroisun{$a$}{$c$}{$b$},\\
\tdtroisdeux{$a$}{$b$}{$c$}^{(321)}&=&\tdun{$a$}\tdun{$b$}\tdun{$c$}-\tdun{$c$}\tddeux{$a$}{$b$}-\tdun{$a$}\tddeux{$b$}{$c$}+\tdtroisdeux{$a$}{$b$}{$c$}.
\end{eqnarray*}
For instance, $\tdtroisdeux{$a$}{$b$}{$c$}^{(132)}$ is a particular example of $\mathcal{T}^\sigma$, with $\mathcal{T}=\tdtroisdeux{$a$}{$b$}{$c$}$
and $\sigma=(132)$. The commutative square implies that $\theta^d(\mathcal{T}^\sigma)$ is the $d$-word $\ell \circ \sigma^{-1}(1)\ldots \ell \circ \sigma^{-1}(n)$. \\

As the edges of the commutative square are Hopf algebra morphisms, lemmas \ref{l3} and \ref{l8} and imply:

\begin{lemma} \label{l3bis} \begin{enumerate}
\item Let $\mathcal{T}_k=(\mathbb{T}_k,\ell_1)$ and $\mathcal{T}_l=(\mathbb{T}_l,\ell_2)$ be two $d$-words of respective degrees $k$ and $l$, 
seen as $d$-decorated trunk trees. For any $(\sigma,\tau,\epsilon) \in \Sigma_k \times \Sigma_l \times Sh(k,l)$:
\BEQ \epsilon^{-1}.\left((\mathbb{T}_k,\ell_1)^\sigma.(\mathbb{T}_l,\ell_2)^\tau\right)=
\sum_{\zeta \in Sh(k,l)} \left(\mathbb{T}_{k+l}, \ell_1 \otimes \ell_2\right)^{\zeta \circ (\sigma \otimes \tau)\circ \epsilon^{-1}}.\EEQ
Here, the action of $\epsilon$ is given by permutations of the decorations.
\item Let $\mathcal{T}=(\mathbb{T}_n,\ell)$ be a $d$-word of degree $n$, seen as a $d$-decorated trunk tree. For any $\sigma \in \Sigma_n$,
\BEQ \Delta(\mathcal{T}^\sigma)=\sum_{k=0}^n \sum_{\substack{\sigma=(\sigma_1 \otimes \sigma_2)\circ \epsilon\\
\sigma_1 \in \Sigma_k,\: \sigma_2 \in \Sigma_{n-k},\: \epsilon \in Sh(k,n-k)}}
\epsilon^{-1}. \left((\mathbb{T}^{\sigma_1},\ell_1) \otimes (\mathbb{T}^{\sigma_2}, \ell_2)\right).  \label{eq:TT2bis} \EEQ
Here also, the action of $\epsilon$ on tensors of decorated forests is given by permutation of the decorations
(the $\epsilon^{-1}$ in the right member comes from the fact that each decoration follows its vertex when we cut the forests of $\mathcal{T}^\sigma$,
so this permutes the letters of $\ell$ according to $\epsilon$).
\end{enumerate}
\end{lemma}

For example:
\begin{eqnarray*}
\tddeux{$a$}{$b$}^{(21)} \tdun{$c$}^{(1)}&=&\tdtroisdeux{$a$}{$b$}{$c$}^{(213)}+\tdtroisdeux{$a$}{$b$}{$c$}^{(312)}
+\tdtroisdeux{$a$}{$b$}{$c$}^{(321)},\\
\Delta\left(\tdtroisdeux{$a$}{$b$}{$c$}^{(321)}\right)&=&1\otimes \tdtroisdeux{$a$}{$b$}{$c$}^{(321)}+\tdun{$c$}\otimes (\tdun{$a$}\tdun{$b$}-\tddeux{$a$}{$b$})
+(\tdun{$b$}\tdun{$c$}-\tddeux{$b$}{$c$}) \otimes \tdun{$a$}+\tdtroisdeux{$a$}{$b$}{$c$}^{(321)}\otimes 1\\
&=&1\otimes \tdtroisdeux{$a$}{$b$}{$c$}^{(321)}+\tdun{$c$}^{(1)}\otimes \tddeux{$a$}{$b$}^{(21)}+\tddeux{$b$}{$c$}^{(21)} \otimes \tdun{$a$}^{(1)}
+\tdtroisdeux{$a$}{$b$}{$c$}^{(321)}\otimes 1\\
&=&1\otimes \tdtroisdeux{$a$}{$b$}{$c$}^{(321)}\otimes 1+(231).\left(\tdun{$a$}^{(1)}\otimes \tddeux{$b$}{$c$}^{(21)}\right)
+(312).\left(\tddeux{$a$}{$b$}^{(21)} \otimes \tdun{$c$}^{(1)}\right)+\tdtroisdeux{$a$}{$b$}{$c$}^{(321)}\otimes 1.
\end{eqnarray*}

\subsection{Applications to iterated integrals}

Let $H$ be a Hopf algebra and $A$  a commutative algebra. Then the set $Char_H(A)$ of algebra morphisms from $H$ to $A$
is a group for the convolution. More precisely, if $\phi,\psi \in Char_H(A)$, then $\phi *\psi=m_A \circ (\phi \otimes \psi) \circ \Delta_H$,
where $m_A$ is the product of $A$ and $\Delta_H$ the coproduct of $H$. The unit of $Char_H(A)$ is $x\mapsto \varepsilon_H(x)1_A$,
where $\varepsilon_H$ is the counit of $H$, and the inverse of $\phi$ is $\phi \circ S_H$, where $S_H$ is the antipode of $H$.

In particular, a character of the shuffle algebra $\Sh^d$ can be seen as a map $\phi$ from the set of $d$-words to $A$, such that $\phi(1)=1_A$ and, for 
any $d$-words $w$ and $w'$:
$$\phi(w)\phi(w')=\sum_{w'' \in Sh(w,w')}\phi(w'').$$
The convolution product of $\phi$ and $\psi$ is given by:
$$(\phi*\psi)(a_1\ldots a_k)=\sum_{i=0}^k \phi(a_1\ldots a_i)\psi(a_{i+1}\ldots a_n).$$
The inverse of the character $\phi$ is $\phi^{-1}=\phi \circ\bar{S}$:
$$\phi^{-1}(a_1\ldots a_n)=(-1)^n \phi(a_n \ldots a_1).$$

A character of $\h^d$ can be seen as a map from the set of $d$-decorated rooted trees to $A$, extended to $\h^d$ by multiplicativity.
The convolution product of $\phi$ and $\psi$ is given by:
$$(\phi*\psi)(\mathbb{F})=\sum_{\vec{v} \models V(\mathbb{F})} \phi(Roo_{\vec{v}}\mathbb{F}) \psi(Lea_{\vec{v}}\mathbb{F}).$$
Seeing $\Sh^d$ as a sub-coalgebra of $\h^d$, this formula for the convolution product also works for $d$-words seen as trunk trees. \\

\begin{defi} The canonical surjection $\pi_{ho}^d$ from $\h_{ho}^d$ to $\h^d$ induces a canonical injection of the group $Char_{\h^d}(A)$ into the 
group $Char_{\h_{ho}^d}(A)$: a character of $\h_{ho}^d$ is a character of $\h^d$ if, and only if, it does not depend of the orders on the vertices of
the ordered forests. In other words, for any $\psi \in Char_{\h_{ho}^d}(A)$, $\psi$ belongs to $Char_{\h^d}(A)$ if, and only if, it is invariant
under forest-order-preserving symmetries, that is to say for any ordered forest $\mathbb{F}$, for any $\sigma \in S_{\mathbb{F}}$,
$\psi(\sigma^{-1}.\mathbb{F})=\psi(\mathbb{F})$.
\end{defi}

As $\theta^d:\h^d \longrightarrow \Sh^d$ is a Hopf algebra morphism, there is a group morphism from $Char_{\Sh^d}(A)$ to $Char_{\h^d}(A)$,
sending a character $\phi$ to $\phi \circ \theta^d$. The character $\bar{\phi}:=
\phi \circ \theta^d$ of $\h^d$ will be called the {\em extension} of $\phi$.\\

For example, let us fix $d$ regular functions $\Gamma_1,\ldots, \Gamma_d$ and $s,t$ in $\mathbb{R}$. For any $d$-word $a_1\ldots a_n$, we let
$$I^{ts}_\Gamma(a_1\ldots a_n)=\int_s^t d\Gamma_{a_1}(x_1)\int_s^{x_1} d\Gamma_{a_2}(x_2) \ldots \int_s^{x_n-1} d\Gamma_{a_n}(x_n).$$
It is well-known that $I^{ts}_\Gamma$ is a character of $\Sh^d$. 
The extension of $I^{ts}_\Gamma$ is defined on any rooted forest $\mathbb{F}$ with decoration $\ell$  in the following way:
\begin{itemize}
\item Choose a heap-order on the vertices of $\mathbb{F}$.
\item Put $[s,t]^\mathbb{F}=\{(x_1,\ldots,x_n)\in [s,t]\:\mid \: \forall i,j,\: 
\left(i \twoheadrightarrow j \mbox{ in }\mathbb{F}\right)\Longrightarrow\left(x_i \leq x_j\right)\}$,
where $n=|\mathbb{F}|$.
\end{itemize}
Then, denoting by $i^-$ the ancestor of the vertex $i$ in $\mathbb{F}$:
\BEA \bar{I}^{ts}_\Gamma(\mathbb{F})&=&\int_{[s,t]^{\mathbb{F}}} d\Gamma_{\ell(1)}(x_1)\ldots d\Gamma_{\ell(n)}(x_n) \nonumber\\
&=& \int_s^t d\Gamma_{\ell(1)}(x_1)\int_s^{x_{2^-}} d\Gamma_{\ell(2)}(x_2)\ldots \int_s^{x_{n^-}} d\Gamma_{\ell(n)}(x_n).
\label{eq:barIts} \EEA

Note that $\bar{I}^{ts}_\Gamma$ is in fact defined on heap-ordered forests, so it is a character of $\h_{ho}^d$.
It is clearly invariant under forest-order-preserving symmetries, so it is a character of $\h^d$.\\

Note that one may alternatively define characters of $\H_{ho}^d$ in the following way. Let $\mu=\mu(dx_1,\ldots,dx_n)$
be some signed measure on $\R^n$, and $\mathbb{F}$ a heap-ordered forest with vertices $1,\ldots,n$. Then one lets
\BEQ I^{ts}_{\mu}(\mathbb{F})=\int_{[s,t]^{\mathbb{F}}} \mu(dx_1,\ldots,dx_n).\EEQ
In particular, letting $\mu_{(\Gamma,\ell)}=d\Gamma_{\ell(1)}\otimes\ldots \otimes d\Gamma_{\ell(n)}$, one has
$I^{ts}_{\Gamma}(\mu_{(\Gamma,\ell)})=I^{ts}_{\mu_{(\Gamma,\ell)}}(\mathbb{F})$,
but the notation $I^{ts}_{\mu}$ allows for an extension from tensor measures $d\Gamma_{\ell(1)}\otimes\ldots\otimes d\Gamma_{\ell(n)}$ to arbitrary measures. \\

As a consequence, if $\phi$ is a character of $\Sh^d$ extended to $\h^d$, and $\mathcal{T}$ is the trunk tree associated to the $d$-word 
$a_1\ldots a_k$, then for any $\sigma \in \Sigma_k$, as $\theta^d(\mathcal{T}^\sigma)=(a_{\sigma^{-1}(1)}\ldots a_{\sigma^{-1}(k)})$:
$$\bar{\phi}(\mathcal{T}^\sigma)=\phi\circ \theta^d(\mathcal{T}^\sigma)=\phi(a_{\sigma^{-1}(1)}\ldots a_{\sigma^{-1}(k)}).$$

This allows to compute easily $\mathbb{T}^\sigma$ and $\mathcal{T}^\sigma$, as explained below. Indeed, if $\mathcal{T}$ is the trunk tree
associated to the $d$-word $a_1\ldots a_k$:
\begin{eqnarray*}
\bar{I}_\Gamma^{st}(\mathcal{T}^\sigma)&=&I_\Gamma^{st}((a_{\sigma^{-1}(1)}\ldots a_{\sigma^{-1}(k)})\\
&=&\int_s^t d\Gamma_{a_{\sigma^{-1}(1)}}(x_{\sigma^{-1}(1)}) \int_s^{x_{\sigma^{-1}(1)}} d\Gamma_{a_{\sigma^{-2}(2)}}(x_{\sigma^{-2}(2)}) 
\ldots \int_s^{x_{\sigma^{-1}(n-1)}} d\Gamma_{a_{\sigma^{-1}(n)}}(x_{\sigma^{-1}(n)}).
\end{eqnarray*}
Using Fubini's theorem, we can write:
$$\bar{I}_\Gamma^{st}(\mathcal{T}^\sigma)=\int_s^t d\Gamma_{a_1}(x_1) \int_{s_2}^{t_2} d\Gamma_{a_2}(x_2)\ldots \int_{s_n}^{t_n}  d\Gamma_{a_n}(x_n),$$
where for all $i$, $s_j \in \{s,x_1,\ldots,x_{j-1}\}$ and $t_j \in \{t,y_1,\ldots, y_{j-1}\}$. Now decompose $\int_{s_i}^{t_i} d\Gamma_{a_i}(x_i)$ into:
$$\left(\int_s^{t_i}-\int_s^{s_i} \right) d\Gamma_{a_i}(x_i).$$
Then $\bar{I}_\Gamma^{st}(\mathbb{T}^\sigma)$ can be written as a sum of terms of the form:
$$\pm \int_s^{\tau_1} d\Gamma_{a_1}(x_1) \int_s^{\tau_2} d\Gamma_{a_2}(x_2) \ldots \int_s^{\tau_n} d\Gamma_{a_n}(x_n),$$
with $\tau_1=t$ and $\tau_i \in \{t,x_1,\ldots,x_{i-1}\}$ for $i=2\ldots n$. Each of these expressions is of the form $\pm I^{st}_\Gamma(\mathbb{F})$
for a particular $d$-decorated rooted forest $\mathbb{F}$, and this gives the expression of $\mathcal{T}^\sigma$,
as for example:
\begin{eqnarray*}
I^{ts}_\Gamma\left(\tdtroisdeux{$a_1$}{$a_2$}{$a_3$}^{(231)}\right)
&=&\int_s^t d\Gamma_{a_3}(x_3) \int_s^{x_3} d\Gamma_{a_1}(x_1) \int_s^{x_1} d\Gamma_{a_2}(x_2)\\
&=&\int_s^t d\Gamma_{a_1}(x_1) \int_s^{x_1} d\Gamma_{a_2}(x_2) \int_{x_1}^t d\Gamma_{a_3}(x_3)\\
&=&\int_s^t d\Gamma_{a_1}(x_1) \int_s^{x_1} d\Gamma_{a_2}(x_2) \int_s^t d\Gamma_{a_3}(x_3)\\
&&-\int_s^t d\Gamma_{a_1}(x_1) \int_s^{x_1} d\Gamma_{a_2}(x_2) \int_s^{x_1} d\Gamma_{a_3}(x_3)\\
&=&I^{ts}_\Gamma(\tddeux{$a_1$}{$a_2$}\: \tdun{$a_3$}\:)-I^{ts}_\Gamma(\hspace{2mm}\tdtroisun{$a_1$}{$a_3$}{\hspace{-1.5mm}$a_2$}\:),
\end{eqnarray*}
so $\tdtroisdeux{$a_1$}{$a_2$}{$a_3$}^{(231)}=\tddeux{$a_1$}{$a_2$}\: \tdun{$a_3$}\:-\hspace{2mm}\tdtroisun{$a_1$}{$a_3$}{\hspace{-1.5mm}$a_2$}\:$.
Choosing three different decorations $a_1=1$, $a_2=2$ and $a_3=3$, we obtain $\mathbb{T}^{(231)}=\tddeux{1}{2} \tdun{3}-\tdtroisun{1}{3}{2}$.\\


\section{Application to rough path theory: the Fourier normal ordering algorithm}

Let $\Gamma=(\Gamma_1(t),\ldots,\Gamma_d(t)):\R\to\R^d$ be some continuous path, compactly
supported in $[0,T]$. Assume that $\Gamma$ is not differentiable, but only 
$\alpha$-H\"older for some $0<\alpha<1$, i.e. bounded in the ${\mathcal C}^{\alpha}$-norm,
\BEQ ||\gamma||_{{\mathcal C}^{\alpha}}:=\sup_{t\in [0,T]} ||\Gamma(t)||+\sup_{s,t\in[0,T]}
\frac{||\Gamma(t)
-\Gamma(s)||}{|t-s|^{\alpha}}.\EEQ
Then iterated integrals of $\Gamma$, $I^{ts}_{\Gamma}(\mathcal{F})$, $\mathcal{F}\in\H^d$, are not
canonically defined. As explained in the Introduction, rough path theory may be seen
as a black box taking as input some lift of $\Gamma$ called {\em rough path over} $\Gamma$,
and producind e.g. solutions of differential equations driven by $\Gamma$. The usual
definition of a rough path is the following. We let in the sequel
 $\lfloor 1/\alpha\rfloor$ be the entire part of $1/\alpha$.

\begin{defi} \label{defiroughpath}

A rough path over $\Gamma$ is a functional $J_{\Gamma}^{ts}(i_1,\ldots,i_n)$, $n\le
\lfloor 1/\alpha\rfloor$, $i_1,\ldots,i_n\in\{1,\ldots,d\}$, such that $J_{\Gamma}^{ts}(i)=
\Gamma_i(t)-\Gamma_i(s)$ are the increments of $\Gamma$, and the following 3 properties
are satisfied:

\begin{itemize}
\item[(i)] (H\"older continuity) $J_{\Gamma}^{ts}(i_1,\ldots,i_n)$ is $n\alpha$-H\"older continuous as a function of two variables, namely, $\sup_{s,t\in\R} \frac{|J_{\Gamma}^{ts}(i_1,\ldots,i_n)|}{|t-s|^{\alpha}}<\infty.$
\item[(ii)] (Chen property) 
\BEQ J_{\Gamma}^{ts}(i_1,\ldots,i_n)=J_{\Gamma}^{tu}(i_1,\ldots,i_n)+
J_{\Gamma}^{us}(i_1,\ldots,i_n)+
\sum_{n_1+n_2=n} J_{\Gamma}^{tu}(i_1,\ldots,i_{n_1})J_{\Gamma}^{us}(i_{n_1+1},\ldots,i_{n});\EEQ

\item[(iii)] (shuffle property)
\BEQ J_{\Gamma}^{ts}(i_1,\ldots,i_{n_1})J_{\Gamma}^{ts}(j_1,\ldots,j_{n_2})=\sum_{\vec{k}\in
Sh(\vec{i},\vec{j})}
J_{\Gamma}^{ts}(k_1,\ldots,k_{n_1+n_2}),\EEQ
where $Sh(\vec{i},\vec{j})$ -- the set of shuffles of the words $\vec{i}$ and $\vec{j}$ -- has been defined
in subsection 1.1.
\end{itemize}

\end{defi}

Axioms (ii) and (iii) may be rewritten in a Hopf algebraic language: indexing
the $J_{\Gamma}^{ts}(i_1,\ldots,i_n)$ by trunk trees $\mathcal{T}$ with decoration
$\ell(j)=i_j$, $j=1,\ldots,n$, they are
equivalent to
\begin{itemize}
\item[{\em (ii)bis}] \BEQ J^{ts}_{\Gamma}(\mathcal{T})=\sum_{\vec{v}\models V(\mathcal{T})}  J^{tu}_{\Gamma}(\Roo_{\vec{v}}(\mathcal{T})) 
J^{us}_{\Gamma}(\Lea_{\vec{v}}(\mathcal{T})),\quad \mathcal{T}\in{\bf Sh}^d;\EEQ 
in other words, $J^{ts}_{\Gamma}=J^{tu}_{\Gamma} \ast J^{us}_{\Gamma}$ for the
shuffle convolution defined in subsection 4.2;
\item[{\em (iii)bis}] \BEQ J^{ts}_{\Gamma}(\mathcal{T})J^{ts}_{\Gamma}(\mathcal{T}')=J^{ts}_{\Gamma}(\mathcal{T} 
\shuffle \mathcal{T}'),
\quad \mathcal{T},\mathcal{T}'\in {\bf Sh}^d.\EEQ In other words, $J^{ts}_{\Gamma}$ is a character
of ${\bf Sh}^d$.
\end{itemize}

Such a functional indexed by {\em trunk trees} extends easily as in subsection 4.2 to a general {\em tree-indexed
functional} or {\em tree-indexed rough path} by setting $\bar{J}^{ts}_{\Gamma}(\mathbb{T}):=
J^{ts}_{\Gamma}\circ\theta^d(\mathbb{T})$, where $\theta^d:{\bf H}^d\to {\bf Sh}^d$ is the canonical
projection morphism defined in subsection 4.1. Since $\theta^d$ is a Hopf algebra
morphism, one gets immediately the generalized properties
\begin{itemize}
\item[{\em (ii)ter}]  $\bar{J}^{ts}_{\Gamma}=\bar{J}^{tu}_{\Gamma} \ast \bar{J}^{us}_{\Gamma}$ for the
 convolution of $\H^d$;
\item[{\em (iii)ter}] $\bar{J}^{ts}_{\Gamma}(\mathcal{T})\bar{J}^{ts}_{\Gamma}(\mathcal{T}')=
\bar{J}^{ts}_{\Gamma}(\mathcal{T}.\mathcal{T}')$, in other
words, $\bar{J}^{ts}_{\Gamma}$ is a character of $\H^d$.
\end{itemize}

Properties (ii), (iii) and their generalizations are satisfied for the usual integration
operators $I^{ts}_{\Gamma}$ and their tree extension $\bar{I}^{ts}_{\Gamma}$, provided
$\Gamma$ is a smooth path so that iterated integrals make sense \cite{Gu2,Kreimer}.

\bigskip

Suppose now one wishes to construct a rough path over $\Gamma$, and concentrate on the
algebraic properties (ii), (iii). Assume one has constructed characters of ${\bf Sh}^d$,
$J^{ts_0}_{\Gamma}$, $t\in [0,T]$ with $s_0$ fixed, such that $J_{\Gamma}^{ts_0}(i)=\Gamma_i(t)-\Gamma_i(s_0)$,
 then one immediately checks that
$J^{ts}_{\Gamma}:=J^{ts_0}_{\Gamma}\ast (J^{ss_0}_{\Gamma}\circ \bar{S})$ satisfies properties (ii)bis and 
(iii)bis. So the only difficult part consists in defining some regularized character of
${\bf Sh}^d$ satisfying the regularity properties (i).

\medskip

For different, intrinsically analytic reasons, it is natural to consider the Fourier
transformed path ${\mathcal F}\Gamma:\R\to\C^d$, $\xi\mapsto \frac{1}{\sqrt{2\pi}}\int_0^T e^{-\II t\xi}
\Gamma(t)dt$, or ${\mathcal F}\Gamma':\xi\mapsto \frac{1}{\sqrt{2\pi}}\int_0^T e^{-\II t\xi} d\Gamma(t)$. On the
one hand, the Besov characterization \cite{Trie} of $\alpha$-H\"older functions uses Fourier
multipliers. On  the other hand, fractional derivation -- given by a convolution operator,
which becomes simply multiplication by some fractional power of $|\xi|$ in Fourier --
is a natural way to generate $\alpha$-H\"older functions, such as e.g. the random paths of
fractional Brownian motion \cite{Alos} or of multifractional processes \cite{Peltier}. In any case, the following
two notions are natural from an analytic point of view:

\begin{defi} \label{defimeassplit}
\begin{itemize}
\item[(i)] (skeleton integral) Let
\BEQ \SkI^t_{\Gamma}(a_1\ldots a_n):=(2\pi)^{-n/2} \int_{\R^n} \prod_{j=1}^n
{\mathcal F}\Gamma'_{a_j}(\xi_j)d\xi_j\ \cdot\ \int^t dx_1\int^{x_1}dx_2\ldots
\int^{x_{n-1}} dx_n e^{\II(x_1\xi_1+\ldots+x_n\xi_n)},\EEQ
where, by definition, $\int^x e^{\II y\xi} dy=\frac{e^{\II x\xi}}{\II\xi}.$ It may be checked
that $\SkI^t_{\Gamma}$ is a character of ${\bf Sh}^d$, just as for usual iterated integrals.

The projection $\Theta^d$ yields immediately a generalization of this notion to tree skeleton
integrals, compare with eq. (\ref{eq:barIts}), if one sees $\overline{Skl}^t_\Gamma$ defined on heap-ordered forests:
\BEA \overline{\SkI}^t_{\Gamma}(\mathbb{T})=\SkI^t_{\Gamma}\circ\Theta^d(\mathbb{T})=
 \int^t d\Gamma_{\ell(1)}(x_1)\int_s^{x_{2^-}} d\Gamma_{\ell(2)}(x_2)\ldots \int^{x_{n^-}} d\Gamma_{\ell(n)}(x_n),\quad
\mathbb{T}\in {\bf H}_{ho}^d
\EEA

An explicit computation yields (\cite{Unterberger}, Lemma 4.5):
\BEQ  \overline{\SkI}^t_{\Gamma}(\mathbb{T})=(2\pi)^{-n/2} \int_{\R^n} \prod_{j=1}^n
{\mathcal F}\Gamma'_{\ell(j)}(\xi_j)d\xi_j\ \cdot\ \frac{e^{\II t(\xi_1+\ldots+\xi_n)}}{\prod_{i=1}^n
[\xi_i+\sum_{j\twoheadrightarrow i} \xi_j]}.\label{E18}\EEQ

As for usual iterated integrals, all this extends to non-tensor measures. Thus one may define
$\overline{\SkI}^t_{\mu}(\mathbb{F})$ if $\mu$ is a signed measure on $\R^n$, and $\mathbb{F}$ a heap-ordered forest with $n$ vertices.

\item[(ii)] (measure-splitting) Let $\mu$ be some signed measure with compact support,
typically, $\mu=\mu_{(\Gamma,\ell)}(dx_1,\ldots,dx_n)=\otimes_{j=1}^n d\Gamma_{\ell(j)}(x_j).$
  Then \BEQ \mu=\sum_{\sigma \in \Sigma_n} \mathcal{P}^\sigma \mu=\sum_{\sigma\in\Sigma_n} \mu^{\sigma}\circ\sigma^{-1},\EEQ
where 
\BEQ {\cal P}^{\sigma}:\mu\mapsto {\cal F}^{-1}\left( {\bf 1}_{|\xi_{\sigma(1)}|\le
\ldots\le |\xi_{\sigma(n)}|} {\cal F}\mu(\xi_1,\ldots,\xi_n) \right) \EEQ
is a Fourier projection, and $\mu^{\sigma}$ is defined by
\BEQ \mu^{\sigma}:={\cal P}^{{\mathrm{Id}}}(\mu\circ\sigma)=({\cal P}^{\sigma}\mu)\circ \sigma. \EEQ

In particular, the following obvious formulas hold:
\BEQ (\mu\circ\eps)^{\sigma}={\cal P}^{\Id}\left( (\mu\circ\eps)\circ\sigma\right)
=\mu^{\eps\circ\sigma},\quad \eps,\sigma\in\Sigma_n; \label{E22}\EEQ
\BEA \mu_1^{\sigma_1}\otimes \mu_2^{\sigma_2}&=& \sum_{\eps\ {\mathrm{shuffle}}} {\cal P}^{\eps}
\left( (\mu_1\otimes\mu_2)\circ(\sigma_1\otimes\sigma_2)\right)
\nonumber\\
&=& \sum_{\eps \ {\mathrm{shuffle}}} \left((\mu_1\otimes\mu_2)\circ(\sigma_1\otimes\sigma_2)\right)^{\eps}\circ\eps^{-1} \nonumber\\
&=& \sum_{\eps\ {\mathrm{shuffle}}} 
(\mu_1\otimes\mu_2)^{(\sigma_1\otimes\sigma_2)\circ\eps}\circ\eps^{-1}. 
\label{eq:musigma1-musigma2}
\EEA

The set of all measures whose Fourier transform is supported in $\{(\xi_1,\ldots,\xi_n);
|\xi_{1}|\le \ldots\le |\xi_{n}| \}$ will be  denoted by ${\cal P}^{\Id} Meas(\R^n)$. Thus $\mu^{\sigma}
\in  {\cal P}^{\Id} Meas(\R^n)$.

\end{itemize}
\end{defi}

To say things shortly, {\em skeleton integrals} are convenient when using Fourier coordinates, 
since they avoid awkward boundary terms such as those generated by  usual
integrals, $\int_0^x e^{\II y\xi}dy=\frac{e^{\II x\xi}}{\II \xi}-\frac{1}{\II\xi}$, 
which create terms with different homogeneity degree in $\xi$ by iterated integrations.
{\em Measure splitting} gives the {\em relative scales} of the Fourier coordinates; orders
of magnitude of the corresponding integrals may be obtained separately in each
sector $|\xi_{\sigma(1)}|\le \ldots\le |\xi_{\sigma(n)}|$. It turns out that these are
easiest to get after a permutation of the integrations (applying Fubini's theorem) such
that {\em innermost (or rightmost)integrals bear highest Fourier frequencies}. This is
the essence of {\em Fourier normal ordering}.

\medskip

Now comes the connection to the preceding sections. Let, for $\mathbb{T}\in {\cal F}_{ho}(n)$,
\BEQ {\cal P}^{\mathbb{T}} Meas(\R^n)=\left\{\mu; \vec{\xi}\in {\mathrm{supp}}({\cal F}\mu)\Rightarrow
 \left( (i\twoheadrightarrow j)\Rightarrow (|\xi_i|>|\xi_j|)  \right) \right\}.\EEQ
This generalizes the spaces ${\cal P}^{\Id} Meas(\R^n)$ defined above for trunk trees.
Assume one finds a way to define
regularized tree-index skeleton integrals $\phi^t_{\mu}(\mathbb{T})$ for every
$\mu\in {\cal P}^{\mathbb{T}} Meas(\R^n)$, such that 
\begin{itemize}
\item[(i)]
$\phi^t_{\mu_{(\Gamma,\ell)}}(\tun)=\SkI^t_{\mu_{(\Gamma,\ell)}}(\tun)=\frac{1}{\sqrt{2\pi}}
\int_{\R} {\mathcal F}\Gamma_{\ell(1)}(\xi) e^{\II t\xi} d\xi$ gives back the original path 
$\Gamma$ for the tree with one single vertex, and
\item[(ii)] if $\mathbb{T}\in {\bf H}_{ho}(n),\mathbb{T}'\in {\bf H}_{ho}(n')$,
and $\mu\in {\cal P}^{\mathbb{T}} Meas(\R^n)$, resp. $\mu'\in {\cal  P}^{\mathbb{T}'} Meas(\R^{n'})$, then 
\BEQ \phi^t_{\mu}(\mathbb{T})
\phi^t_{\mu'}(\mathbb{T}')=\phi^t_{\mu\otimes\mu'}(\mathbb{T}.\mathbb{T}'),\EEQ
which is an extension of the multiplicative property defining a character of $\H^d_{ho}$.
\end{itemize}

\medskip

Then one is tempted to think that, interpreting $\phi^t_{\mu^{\sigma}_{(\Gamma,\ell)}}(\mathbb{T})$
as coming from a  rough path $J^t_{\mu_{(\Gamma,\ell)}}$ over $\Gamma$ by the above defined measure-splitting
procedure, the following sequence of postulated equalities, where $\mathbb{T}_n\in {\bf H}_{ho}$ is the trunk tree with $n$ vertices,
\BEA J^t_{\mu_{(\Gamma,\ell)}}(\mathbb{T}_n)&=& \sum_{\sigma\in\Sigma_n} J^t_{\mu^{\sigma}_{(\Gamma,\ell)}\circ\sigma} (\mathbb{T}_n) \quad
{\mathrm{by\ linearity}} \nonumber\\
&=& \sum_{\sigma\in\Sigma_n} J^t_{\mu^{\sigma}_{(\Gamma,\ell)}}(\mathbb{T}^{\sigma}) \quad
{\mathrm{by\ the \ results\ of \ subsection\ 4.2}}\nonumber\\
&=&  \sum_{\sigma\in\Sigma_n} \phi^t_{\mu^{\sigma}_{(\Gamma,\ell)}}(\mathbb{T}^{\sigma}) 
\EEA
define a character of ${\bf Sh}^d$, and thus allow (leaving aside the regularity
properties which must be checked independently) to define a rough path over $\Gamma$.
This is the content of the following Lemma, which is actually, as we shall see, equivalent
to stating that $\Theta$ is a Hopf algebra morphism:

\begin{defi} \label{def:barchi-char}

\begin{itemize}

\item[(i)] Let $\phi^t_{\mathbb{T}}: {\cal P}^{\mathbb{T}}Meas(\R^n)\to\R, \mu\mapsto \phi^t_{\mathbb{T}}(\mu)$, also written $\phi^t_{\mu}(\mathbb{T})$ 
$(t\in\R, \mathbb{T}\in {\bf H}_{ho}(n))$  be a family of linear forms such that
\begin{itemize}
\item[(a)]
$\phi^t_{\tun}(\mu_{(\Gamma,i)})-\phi^s_{\tun}(\mu_{(\Gamma,i)})=\Gamma_t(i)-\Gamma_s(i)$;
\item[(b)] 
 if $(\mathbb{T}_i,\mu_i)\in
\H_{ho}(n_i)\times {\cal P}^{\mathbb{T}_i} Meas(\R^{n_i})$, $i=1,2$, the following 
{\em ${\bf H}_{ho}$-multiplicative property} holds,
\BEQ \phi^t_{\mu_1}(\mathbb{T}_1)  \phi^t_{\mu_2}(\mathbb{T}_2)=\phi^t_{\mu_1\otimes\mu_2}(\mathbb{T}_1.\mathbb{T}_2);
\label{eq:T1wedgeT2}\EEQ
\item[(c)] $\phi^t$ is  invariant under forest-ordering preserving symmetries, that is to say:
\BEQ \phi^t_\mu(\mathbb{F})=\phi^t_{\mu \circ \sigma}(\sigma^{-1}.\mathbb{F})\mbox{ if } \sigma \in S_{\mathbb{F}}.\label{E28} \EEQ
\end{itemize}

\item[(ii)] Let, for $\Gamma=(\Gamma(1),\ldots,\Gamma(d))$,
 ${\chi}_{\Gamma}^t:{\bf Sh}^d\to\R$ be the linear form on ${\bf Sh}^d$ defined by
\BEQ {\chi}_{\Gamma}^t (\mathbb{T}_n,\ell):=\sum_{\sigma\in \Sigma_n} \phi^t_{\mu_{(\Gamma,\ell)}^{\sigma}}(\mathbb{T}^{\sigma}).
\label{E29}\EEQ
\end{itemize}
\end{defi}

\begin{lemma} \label{lem:barchi-char}
\begin{enumerate}

\item For every path $\Gamma$ such that $\chi_{\Gamma}^t$ is well-defined,
 ${\chi}_{\Gamma}^t$ is a character of ${\bf Sh}^d$.

Consequently, the following formula  for $\mathbb{T}_n\in {\bf Sh}^d$,  $n\ge 1$, with $n$ vertices and decoration $\ell$,
  \BEQ (J')^{ts}_{\Gamma}(\ell(1)\ldots\ell(n)):=
   {\chi}^t_{\Gamma} \ast ({\chi}^s_{\Gamma}\circ S)
(\mathbb{T}_n) \EEQ
defines a rough path over $\Gamma$.

\item  $(J')^{ts}_{\Gamma}(\ell(1)\ldots \ell(n))=J^{ts}_{\Gamma}(\ell(1)\ldots \ell(n))$ where
\BEQ  J^{ts}_{\Gamma}(\mathbb{T}_n):=\sum_{\sigma\in\Sigma_n} \left(
\phi^t \ast (\phi^s\circ \bar{S})\right)_{\mu^{\sigma}_{(\Gamma,\ell)}}(\mathbb{T}^{\sigma}). \EEQ

The convolution $\phi^t \ast (\phi^s\circ\bar{S})$ in the last formula is defined
by reference to the ${\bf H}_{ho}$-coproduct, namely, one sets 
\BEQ (\phi^t\ast(\phi^s\circ\bar{S}))_{\nu}(\mathbb{T})=\sum_{\vec{v}\models V(\mathbb{T})} \phi^t_{
\otimes_{v\in V(\Roo_{\vec{v}}\mathbb{T})} \nu_v} (\Roo_{\vec{v}}\mathbb{T})  \phi^s_{
\otimes_{v\in V(\Lea_{\vec{v}}\mathbb{T})} \nu_v} (\bar{S}(\Lea_{\vec{v}}\mathbb{T})) \EEQ
for a tensor measure $\nu=\nu_1\otimes\ldots\otimes\nu_n$, and by multilinear extension

\BEA &&  \left(\phi^t\ast(\phi^s\circ\bar{S})\right)_{\nu}(\mathbb{T})=(2\pi)^{-n/2}
\int {\cal F}\nu(\xi_1,\ldots,\xi_n)d\xi_1\ldots d\xi_n \ \cdot\nonumber\\
&& \cdot 
\sum_{\vec{v}\models V(\mathbb{T})} \phi^t_{\otimes_{v\in V(Roo_{\vec{v}}(\mathbb{T}))} e^{\II  x_v\xi_v}
dx_v}(\Roo_{\vec{v}}\mathbb{T})   \phi^s_{\otimes_{v\in V(Lea_{\vec{v}}(\mathbb{T}))} e^{\II  x_v\xi_v}
dx_v}(\bar{S}(\Lea_{\vec{v}}\mathbb{T})) ,\quad \mathbb{T}\in {\cal F}_{ho}(n).  \nonumber\\ 
\label{convol} \EEA 

for an arbitrary measure $\nu\in Meas(\R^n)$.  
\end{enumerate}
\end{lemma}

{\bf Proof.} 
\begin{enumerate}
\item
  Let $\mathbb{T}_{n_i}\in {\bf H}_{ho}$ be the heap-ordered trunk trees  
with $n_i$ vertices, $i=1,2$; define $n:=n_1+n_2$.  All right-, resp. left shuffles
 $\eps,\zeta$ below are
intended to be shuffles of $(1,\ldots,n_1),(n_1+1,\ldots,n_2)$.   Then, with $\ell=\ell_1\otimes \ell_2$ and writing $\mu$ instead of $\mu_{(\Gamma,\ell)}$:

\BEA \chi^t_{\Gamma}(({\mathbb{T}}_{n_1},\ell_1)\shuffle ({\mathbb{T}}_{n_2},\ell_2)) 
&=& \sum_{\zeta} \chi^t_{\Gamma}(({\mathbb{T}}_n,\ell\circ\zeta))\mbox{ by (\ref{E29})}\nonumber\\
&=& \sum_{\zeta,\sigma} \phi^t_{\mu^{\sigma}_{(\Gamma,\ell\circ\zeta)}}(\mathbb{T}^{\sigma})\nonumber\\
&=&\sum_{\zeta,\sigma} \phi^t_{(\mu\circ\zeta)^{\sigma}}(\mathbb{T}^{\sigma}) \nonumber\\
&=& \sum_{\zeta,\sigma} \phi^t_{\mu^{\zeta\circ\sigma}}(\mathbb{T}^{\sigma}) \mbox{ by (\ref{E22})} \nonumber\\
&=& \sum_{\zeta,\sigma} \phi^t_{\mu^{\sigma}}(\mathbb{T}^{\zeta^{-1}\circ\sigma}). \label{E34}  \EEA

On the other hand, denoting $\mu_i=\mu_{(\Gamma,\ell_i)}$:
\BEA && \chi^t_{\Gamma}(({\mathbb{T}}_{n_1},\ell_1))\chi^t_{\Gamma}(({\mathbb{T}}_{n_2},\ell_2)) \nonumber\\
&&=\sum_{\sigma_1,\sigma_2} \phi^t_{\mu_1^{\sigma_1}}(\mathbb{T}^{\sigma_1})\phi^t_{\mu_2^{\sigma_2}}(\mathbb{T}^{\sigma_2}) \nonumber\\
&&=\sum_{\sigma_1,\sigma_2,\eps} \phi^t_{\mu^{(\sigma_1\otimes\sigma_2)\circ\eps}\circ
\eps^{-1}}(\mathbb{T}^{\sigma_1}.\mathbb{T}^{\sigma_2}) \mbox{ by (\ref{eq:musigma1-musigma2})} \nonumber\\
&&= \sum_{\sigma_1,\sigma_2,\eps} \phi^t_{\mu^{(\sigma_1\otimes\sigma_2)\circ\eps}}
(\eps^{-1}(\mathbb{T}^{\sigma_1}.\mathbb{T}^{\sigma_2})) \nonumber\\
&&= \sum_{\sigma_1,\sigma_2,\eps,\zeta}  \phi^t_{\mu^{(\sigma_1\otimes\sigma_2)\circ\eps}}
\left(\mathbb{T}^{\zeta^{-1}\circ(\sigma_1\otimes\sigma_2)\circ\eps}\right) \mbox{ by lemma \ref{l8}} \nonumber\\
&&= \sum_{\zeta,\sigma} \phi^t_{\mu^{\sigma}}(\mathbb{T}^{\zeta^{-1}\circ\sigma}). \label{E35}  \EEA

\item 

 Let us check that, for a tensor measure $\mu\in Meas(\R^n)$,
$$\sum_{\sigma\in\Sigma_n} (\phi^t \ast (\phi^s\circ \bar{S}))_{\mu^{\sigma}}(\mathbb{T}^{\sigma})=
(\chi^t \ast (\chi^s\circ \bar{S}))_\mu({\mathcal{T}}_n).$$

Assume ${\cal F}\mu(\xi_1,\ldots,\xi_n)=\delta(\xi-\xi^0)$ is a Dirac distribution, where
$|{\xi^0}_{\sigma_0(1)}|<\ldots<|{\xi^0}_{\sigma_0(n)}|$ for some $\sigma_0\in\Sigma_n$.
By construction, $\mu^{\sigma}=0$ unless $\sigma=\sigma_0$, so 
\BEQ \sum_{\sigma\in\Sigma_n} (\phi^t \ast (\phi^s\circ \bar{S}))_{\mu^{\sigma}}(\mathbb{T}^{\sigma})=
(\phi^t \ast (\phi^s\circ \bar{S}))_{\mu^{\sigma_0}}(\mathbb{T}^{\sigma_0}). \EEQ
We now apply the coproduct formula (\ref{eq:TT2bis}) to $\mathcal{T}^{\sigma_0}$, with $\mathcal{T}=(\mathbb{T}_n,\ell)$, such that $\ell(i)=i$
for all $i$. We obtain:
\BEQ \Delta(\mathcal{T}^{\sigma_0})=\sum_{k=0}^n \sum_{\sigma_0=(\sigma_1\otimes\sigma_2)\circ\eps}
\epsilon^{-1}.(\left(\mathbb{T}^{\sigma_1},\ell_1)\otimes(\mathbb{T}^{\sigma_2},\ell_2)\right).\EEQ

For fixed $k \in \{0,\ldots,n\}$, let us write $\mu=\mu_1 \otimes \mu_2$, with $\mu_1\in Meas(\R^k)$ and $\mu_2 \in Meas(\R^{n-k})$.
The shuffle $\eps$ may be made to act on $\mu^{\sigma_0}$ instead of $(\mathbb{T}^{\sigma_1},\ell_1)\otimes(\mathbb{T}^{\sigma_2},\ell_2)$
because of the invariance condition, resulting in the product measure $\mu^{\sigma_0}\circ\eps^{-1}=\mu_1^{\sigma_1}\otimes \mu_2^{\sigma_2}$,
see eq. (\ref{eq:musigma1-musigma2}), so the convolution in (\ref{convol}) writes simply 
\BEA && \sum_{k=0}^n \sum_{\sigma_0=(\sigma_1\otimes\sigma_2)\circ\eps}
\phi^t_{\mu_1^{\sigma_1}}(\mathbb{T}^{\sigma_1}) (\phi^s_{\mu_2^{\sigma_2}}\circ\bar{S})(\mathbb{T}^{\sigma_2}) \nonumber\\
 && = \sum_{k=0}^n \sum_{\sigma_1,\sigma_2} \phi^t_{\mu_1^{\sigma_1}}(\mathbb{T}^{\sigma_1})
(\phi^s_{\mu_2^{\sigma_2}}\circ\bar{S})(\mathbb{T}^{\sigma_2})  \nonumber\\
&& = (\chi^t\ast (\chi^s\circ \bar{S}))_{\mu}({\mathbb{T}}_n), \EEA 
where $\mathbb{T}_n$ is the trunk tree with $n$ vertices.

\end{enumerate} \hfill \eop

Conversely, assume one has some  path-dependent shuffle character
 $\Gamma\rightsquigarrow \chi^t_{\Gamma}(\mathcal{T})$ for every $\Gamma$, which (when extended to a measure-indexed character) is linear; $\chi^t_{\Gamma}\ast(\chi^s_{\Gamma}\circ S)$ is  then a rough path over $\Gamma$. 
One may define  $\phi^t(\mathbb{T}^{\sigma})$, $\sigma\in\Sigma_n$ from eq. (\ref{E29}), and then $\phi^t(\mathbb{T})$ for an arbitrary tree since permutation graphs generate ${\bf H}_{ho}$. Eq. (\ref{E28}) is then
trivially satisfied, and we claim that  eq. (\ref{eq:T1wedgeT2}) also holds. Namely, circulating through eq. (\ref{E34}) and
(\ref{E35}), and extending $\chi^t$ to a projected measure $\mu=(\mu_1^{\sigma_1}\circ \sigma_1^{-1})\otimes
(\mu_2^{\sigma_2}\circ \sigma_2^{-1})$, one gets
\BEA &&  \phi^t_{\mu_1^{\sigma_1}}(\mathbb{T}^{\sigma_1})\phi^t_{\mu_2^{\sigma_2}}(\mathbb{T}^{\sigma_2})=\chi^t_{\mu_1^{\sigma_1}}(({\mathbb{T}}_{n_1},\ell_1))\chi_{\mu_2^{\sigma_2}}(({\mathbb{T}}_{n_2},\ell_2))=
\chi^t_{\mu_1^{\sigma_1}\otimes\mu_2^{\sigma_2}}(({\mathbb{T}}_{n_1},\ell_1)\shuffle ({\mathbb{T}}_{n_2},\ell_2))  \nonumber\\
&&=\sum_{\zeta,\sigma} \phi^t_{\mu^{\sigma}}(\mathbb{T}^{\zeta^{-1}\circ\sigma})= \sum_{\sigma_1,\sigma_2,\eps} \phi^t_{\mu^{(\sigma_1\otimes\sigma_2)\circ\eps}}
(\eps^{-1}(\mathbb{T}^{\sigma_1}.\mathbb{T}^{\sigma_2})) \nonumber\\
&&= \phi^t_{\mu_1^{\sigma_1}\otimes\mu_2^{\sigma_2}}(\mathbb{T}^{\sigma_1}.\mathbb{T}^{\sigma_2}).\EEA Hence all axioms of Definition \ref{def:barchi-char}
hold.

In this sense the Fourier normal ordering algorithm for constructing rough paths yields {\em all} possible formal
rough paths.

\bigskip

It is easy to realize what one has gained by this approach. One has essentially reduced
the problem of {\em regularizing characters of the shuffle algebra ${\bf Sh}^d$}
 to that of {\em regularizing
 characters of the Connes-Kreimer algebra ${\bf H}^d$}. Now a character of ${\bf H}^d$ is
simply the linear and multiplicative extension of an {\em arbitrary} function on the
generating set ${\bf T}$. The only (obvious) constraint is that $\phi^t_{d\Gamma_i}(\tun)=\SkI^t_{\Gamma}(\tdun{$i$})$ up to an integration constant
 for trees with one vertex, so that
the above construction yields a rough path {\em over $\Gamma$}. Regularization schemes 
respecting the multiplicative structure of $\bf H$ abund in the literature under the
name of {\em renormalization schemes}, from the most abstract ones involving the 
Rota-Baxter equation, to the renormalization schemes in quantum field theory obtained
by  recursively defined counterterms, constructed e.g. by
letting external momenta of Feynman diagrams go to zero. Indeed, such constructions
allow to construct rough paths \cite{Unt-ren}; the first rough path construction \cite{Unterberger,Unterberger-fbm} by Fourier
normal ordering (a domain regularization obtained by simply cutting conical regions surrounding the dangerous Fourier directions
where the denominator in (\ref{E18}) is small) is of a different type though. 
On the other hand, it seems difficult to try and find directly counterterms that are multiplicative
{\em with respect to the shuffle product}.


\end{document}